\documentclass[final,onefignum,onetabnum]{siamart220329}

\usepackage{graphicx, epstopdf}
\usepackage{amsmath, amsfonts, mathtools, dsfont}
\usepackage{color, paralist, booktabs}
\usepackage{algorithmic}
\ifpdf
  \DeclareGraphicsExtensions{.eps,.pdf,.png,.jpg}
\else
  \DeclareGraphicsExtensions{.eps}
\fi

\newsiamremark{remark}{Remark}
\newsiamremark{hypothesis}{hypothesis}
\crefname{hypothesis}{hypothesis}{Hypotheses}
\newsiamthm{claim}{Claim}

\headers{Numerical computation of spiral spectra}{S. Dodson, R. Goh, and B. Sandstede}
\title{Efficient numerical computation of spiral spectra with exponentially-weighted preconditioners\thanks{Submitted to the editors May, 9th 2024.
\funding{Sandstede acknowledges support by the NSF through grants 2038039 and 2106566. Goh acknowledges support by the NSF through grants DMS-2006887, DMS-2307650}}}

\author{
Stephanie Dodson\thanks{Mathematics Department, Colby College, Waterville, USA}
\and
Ryan Goh\thanks{Department of Mathematics and Statistics, Boston University, Boston, USA}
\and
Bj\"orn Sandstede\thanks{Division of Applied Mathematics, Brown University, Providence, USA}
}
\usepackage{amsopn}

\graphicspath{{Figures/}}

\newcommand{\C}{\mathbb{C}}

\newcommand{\R}{\mathbb{R}}
\newcommand{\Z}{\mathbb{Z}}

\def\Re{\mathop\mathrm{Re}\nolimits}    % real part
\newcommand{\rmO}{\mathrm{O}}

\newcommand{\rmd}{\mathrm{d}}
\newcommand{\rme}{\mathrm{e}}
\newcommand{\rmi}{\mathrm{i}}
\newcommand{\Rg}{\mathrm{Rg}}

\newcommand{\spec}{\mathrm{spec}}

\begin{document}

\maketitle

\begin{abstract}
The stability of nonlinear waves on spatially extended domains is commonly probed by computing the spectrum of the linearization of the underlying PDE about the wave profile. It is known that convective transport, whether driven by the nonlinear pattern itself or an underlying fluid flow, can cause exponential growth of the resolvent of the linearization as a function of the domain length. In particular, sparse eigenvalue algorithms may result in inaccurate and spurious spectra in the convective regime. In this work, we focus on spiral waves, which arise in many natural processes and which exhibit convective transport. We prove that exponential weights can serve as effective, inexpensive preconditioners that result in resolvents that are uniformly bounded in the domain size and that stabilize numerical spectral computations. We also show that the optimal exponential rates can be computed reliably from a simpler asymptotic problem posed in one space dimension.
\end{abstract}


% REQUIRED

%%%%%%%%%%%%%%%%%%%%%%%%%%%%%%%%%%%%%%%%%%%%%%%%%%%%%%%%%%%%%%%%%%%%%%%%%

\section{Introduction}\label{s:intro}

Spatiotemporal patterns arise in many natural and physical systems across vast scales. Examples include vegetation patterns in semi-arid environments \cite{Dawes:2016cz,Siteur:2014jm} and mussel beds \cite{Liu}, oscillating chemical reactions \cite{Winfree:1972ud,winfree}, and traveling-wave patterns of electrical activity in neurons and cardiac dynamics \cite{Marcotte1,Karma1,Karma2,Alonso:2016et}. Investigating the formation and stability of these patterns can provide insight into their specific roles and an enhanced understanding of the system. 

Spatiotemporal patterns are commonly studied in reaction-diffusion systems of the form $\mathbf{u}_t = D \Delta \mathbf{u} + f(\mathbf{u})$ where $\mathbf{u} = \mathbf{u}(\mathbf{x},t) \in \mathbb{R}^n$, the smooth nonlinearity $f$ represents local dynamics, and spatial coupling is mediated through the Laplacian. These equations have been studied both on the unbounded domain $\mathbf{x} \in \mathbb{R}^{2}$ and on bounded domains $\mathbf{x} \in [0,L]^2$ of length $L$ coupled with appropriate boundary conditions. Patterns that rotate or travel uniformly in time are stationary in appropriate co-rotating or co-moving frames and can therefore be computed efficiently and accurately through numerical root-finding schemes. To characterize the stability properties of such patterns, it is often informative, and in many cases sufficient, to compute the spectrum of the linearization $\mathcal{L}$ of the model system evaluated at the patterned state.

The numerical computation of the spectrum of $\mathcal{L}$ is not always straightforward though. It is well documented that a differential operator $\mathcal{L}$ posed on the one-dimensional domain $x\in[0,L]$ with $L\gg1$ large can exhibit spurious eigenvalues when the norm of its resolvent $(\mathcal{L}-\lambda)^{-1}$ grows as $L$ increases due to numerical instabilities. This phenomenon was investigated, for instance, for constant-coefficient advection-diffusion operators in \cite{trefethen1,trefethen2} via the notion of pseudospectra \cite{trefethen3}. In \cite{ss-trunc}, the lower bound $\|(\mathcal{L}-\lambda)^{-1}\|_{L^2(0,L)}\geq\rme^{\eta(\lambda)L}$ was established for operators with asymptotically constant coefficients, where the exponential rate $\eta(\lambda)$ was linked explicitly to the spatial eigenvalues $\nu(\lambda)$ of the matrix $A(x;\lambda)$ that arises when rewriting the eigenvalue problem $(\mathcal{L}-\lambda)\mathbf{u}=0$ as a first-order spatial dynamical system $\frac{\rmd}{\rmd x} \mathbf{v} = A (x;\lambda) \mathbf{v}$. The recent numerical computations published in \cite{groot} showed that the use of exponential weights of the form $\rme^{-\eta(\lambda)x}$ stabilizes eigenvalue computations of fluid-flow eigenvalue problems with asymptotically constant coefficients on channel-like domains $\mathbf{x}=(x,\mathbf{y})\in[0,L]\times\Omega$ with $\Omega$ bounded and $L\gg1$ large as suggested in \cite{trefethen1,trefethen2,ss-trunc}.

Of interest to us are spiral-wave patterns. These nonlinear waves arise in applications including oscillating chemical reactions of the Belousov--Zhabotinsky reaction \cite{Belmonte,Jahnke1989,Winfree:1972ud} and in cAMP signaling in cellular slime molds \cite{Newell:1982jr}, and they have also been linked to abnormal cardiac rhythms \cite{Marcotte1,Karma1,Karma2,Alonso:2016et}. Spiral waves have thus been the subject of a host of analytical, numerical, and experimental studies; see, for example, \cite{Perez-Munuzuri1991,Diercks,Biktasheva,Barkleyspectra,BKT,ss-mem, BW} and references therein.

A rigidly-rotating spiral wave has a fixed spatial profile that converges to a periodic wave-train in the far field away from the core and rotates in time with a constant temporal frequency. Hence, spiral waves are stationary in appropriate co-rotating coordinate frames. Their stability on bounded disks $B_R(0)$ can be understood via the spectrum of the linearization $\mathcal{L}_R$. In particular, many instabilities, including transitions to meander and drift, period-doubling bifurcations, and spiral-wave break-up, have been shown to be caused by eigenvalues (see \cite[\S12]{ss-mem} for an overview of these phenomena and further references), and it is therefore important to understand how reliable numerical eigenvalue computations are for $\mathcal{L}_R$.

The computation of eigenvalues of $\mathcal{L}_R$ is challenging for even moderate values of the radius $R$, since  convective transport on the unbounded plane towards the far field manifests itself as growth of the resolvent of the non-normal operator $\mathcal{L}_R$ as $R$ increases. While it is known that, with the exception of a discrete set of eigenvalues, the spectrum of $\mathcal{L}_R$ converges to a collection of algebraic curves, termed the absolute spectrum $\Sigma_\mathrm{abs}$, as the radius $R$ grows \cite{ss-trunc,ss-spst,ss-mem}, computations often paint a very different picture. As the domain radius increases, the spectrum appears to approach a different set of curves, given by the essential spectrum of the unbounded-domain linearization, that is distinct from the theoretically predicted limit. This unexpected eigenvalue behavior is caused by the large resolvent norm. Given the relevance of eigenvalues for spiral instabilities, it is therefore important to be able to extract eigenvalues reliably from spectral computation. In other words, we need to understand when we can, and cannot, trust numerical eigenvalue computations in this context.

In this paper, we demonstrate that the spectra of spiral waves can be computed accurately by using preconditioners that consist of exponential weights of the form $\rme^{\eta(\lambda)|\mathbf{x}|}$. Notably, Theorem~\ref{p:rg} characterizes (1) the nonempty set of $\lambda$ for which the resolvent grows exponentially with the lower bound $\|(\mathcal{L}_R-\lambda)^{-1}\|_{L^2(B_R(0),\R^N)} \geq \rme^{\eta(\lambda)R}$ for some $\eta(\lambda)>0$, and (2) the set of $\lambda$ for which the resolvent is  bounded uniformly in $R$. Theorem~\ref{t:1} shows that the resolvent is bounded uniformly in $R$ with $\|(\rme^{-\eta(\lambda)|\mathbf{x}|} \mathcal{L}_R \rme^{\eta(\lambda)|\mathbf{x}|} - \lambda)^{-1}\|_{L^2(B_R(0),\R^N)} \leq C$ when posed on an appropriate exponentially weighted space. Furthermore, we show how the rates $\eta(\lambda)$ can be calculated accurately and efficiently from the spatial eigenvalues of the asymptotic far-field operator: the resulting exponential weights therefore serve as inexpensive preconditioners.

The paper is outlined as follows. We review the case of convection-diffusion operators in Section~\ref{s:1d} to illustrate the relevant mathematical terminology, techniques, and phenomena. The necessary background on spiral waves, their spectra, and the statements of the main results are presented in Section~\ref{s:mr} and their proofs in Section~\ref{s:proof}. In Section~\ref{s:num}, we demonstrate that the proposed use of exponential weights as preconditioners indeed facilitates the accurate numerical computation of spiral spectra in the Barkley model. We emphasize that, while our main results are stated for spiral waves, the presented numerical algorithm can be deployed also in other applications as demonstrated by the convection-diffusion operator and the work in \cite{groot}.

%%%%%%%%%%%%%%%%%%%%%%%%%%%%%%%%%%%%%%%%%%%%%%%%%%%%%%%%%%%%%%%%%%%%%%%%%

\section{Review: Convection-diffusion operators}\label{s:1d}

To motivate our results, we illustrate the phenomena of interest using the often studied convection-diffusion operator $\mathcal{L}_Ru := u_{xx} + c u_x$ for positive drift speed $c>0$ on large intervals $x\in (-R/2,R/2)$ of length $R\gg1$ with Dirichlet boundary conditions $u|_{x=\pm R/2} = 0$. The results described here can be found in \cite{ss-trunc,trefethen2}\footnote{We remark that the results in \cite{ss-trunc} apply more generally to differential operators of order $n$ with asymptotically constant coefficients and arbitrary separated boundary conditions.}, and we therefore keep the discussion mostly informal. The spectrum of $\mathcal{L}_R$ is given by $\Sigma_R=\{-\frac{c^2}{4}-\frac{n^2\pi^2}{R^2}\colon n\in \mathbb{N}\}$. As $R\to\infty$, the set $\Sigma_R$ converges locally uniformly to the \emph{absolute spectrum} $\Sigma_\mathrm{abs}=\{\lambda\in\C\colon \lambda\leq -\frac{c^2}{4}\}$ in the symmetric Hausdorff distance. Next, we consider the spectrum of $\mathcal{L}_\infty$ posed on the whole line $\R$, which can be analysed by writing the eigenvalue problem $\mathcal{L}_\infty u = \lambda u$ as the first-order spatial dynamical system
\[
\frac{\rmd}{\rmd x}\begin{pmatrix} u \\ v \end{pmatrix} = A(\lambda) \begin{pmatrix} u \\ v \end{pmatrix}, \qquad
A(\lambda) = \begin{pmatrix} 0 & 1 \\ 0 & -c \end{pmatrix}.
\]
The eigenvalues $\nu(\lambda)$ of $A(\lambda)$, often referred to as \emph{spatial eigenvalues}, satisfy the dispersion relation $\lambda = \nu^2 + c \nu$. We order them by real part, with $\Re \nu_{-1}(\lambda)<0<\Re \nu_0(\lambda)$ for $\lambda>0$, so that
$\nu_{-1}(\lambda) = -\frac{c}{2} -\sqrt{\frac{c^2}{4} + \lambda}$ and
$\nu_{0}(\lambda)  = -\frac{c}{2} +\sqrt{\frac{c^2}{4} + \lambda}$,
and define the spectral gap $J_0(\lambda) = (-\Re \nu_0(\lambda),-\Re \nu_{-1}(\lambda))\subset \R$. The spatial eigenvalues can be used to characterize both the absolute spectrum via
\[
\Sigma_\mathrm{abs}
= \left\{\lambda\in\C\colon \Re\nu_0(\lambda)=\Re\nu_{-1}(\lambda) \right\}
= \left\{\lambda\in\C\colon J_0(\lambda) = \emptyset \right\}
\]
and the Fredholm boundary $\Sigma_\mathrm{FB}$ of $\mathcal{L}_\infty$ posed on $L^2(\R)$ via
\[
\Sigma_\mathrm{FB}
= \left\{\lambda\in\C\colon \nu_0(\lambda)\in\rmi\R \right\}
= \left\{\lambda\in\C\colon \lambda = -\ell^2 + \rmi c \ell, \; \ell\in\R \right\}.
\]
Since the operator has constant coefficients, $\Sigma_{FB}$ is equal to the essential spectrum. Instead of the usual $L^2$ space, we can also pose $\mathcal{L}_\infty$ on the exponentially-weighted function space  $L_\eta^2(\R,\C):=\{ u\in L^2_\mathrm{loc}\colon |u|_{L^2_\eta}:=|u(x)\rme^{\eta x}|_{L^2}<\infty\}$ with $\eta\in\R$, which is equivalent to considering the conjugated operator $\mathcal{L}_{\infty,\eta}:=\rme^{\eta x}\mathcal{L}_\infty\rme^{-\eta x}$ on $L^2(\R)$. The Fredholm boundary $\Sigma_{\mathrm{FB},\eta}$ of $\mathcal{L}_{\infty,\eta}$ is given by
\[
\Sigma_{\mathrm{FB},\eta}
= \left\{\lambda\in\C\colon \nu_0(\lambda)-\eta\in\rmi\R \right\}
= \left\{ \lambda\in\C\colon \lambda = -\ell^2 + \rmi\ell(c- 2\eta)+ \eta^2 - c\eta,\; \ell\in\mathbb{R} \right\}.
\]
In particular, the spectrum is shifted to the left for weights $\eta$ with $0<\eta\leq c/2$.

For $R\gg1$, the works \cite{trefethen2, ss-trunc} show that the resolvent operator $(\mathcal{L}_R-\lambda)^{-1}$ is bounded uniformly in $R\gg1$ for each $\lambda$ to the right of $\Sigma_\mathrm{FB}$, that is, for all $\lambda$ for which $0\in J_0(\lambda)$. In addition, these papers show that the norm of $(\mathcal{L}_R-\lambda)^{-1}$ grows exponentially in $R$ for $\lambda$ to the left of $\Sigma_\mathrm{FB}$.

\begin{proposition}[{\cite[Thms~5 \& 7]{trefethen2}, \cite[Prop~2]{ss-trunc}}] \label{p:cd-r}
Let $\lambda_*\in\C\setminus\Sigma_\mathrm{abs}$ with $0\not\in \overline{J_0(\lambda_*)}$ so that $\Re\nu_{-1}(\lambda_*)<\Re\nu_{0}(\lambda_*)<0$, then there are constants $\delta,C,R_*>0$ so that 
$\|(\mathcal{L}_R-\lambda)^{-1}\|_{L^2(-R/2,R/2)} \geq C \rme^{|\nu_0(\lambda_*)|R}$
uniformly in $R\geq R_*$ for all $\lambda\in B_\delta(\lambda_*)$.
\end{proposition}

Furthermore, it was shown in \cite{ss-trunc} (and this can also be inferred from the results in \cite{trefethen2}) that the resolvent stays bounded uniformly in $R$ provided it is posed on $L^2_\eta$ for an appropriate weight $\eta$.

\begin{proposition}[{\cite[Prop~1]{ss-trunc}}] \label{p:cd-eta}
Let $\lambda_*\in\C\setminus\Sigma_\mathrm{abs}$ with $0\not\in\overline{J_0(\lambda_*)}$ and fix $\eta\in J_0(\lambda_*)$, then there are constants $\delta,C,R_*>0$ so that 
$\|(\mathcal{L}_R-\lambda)^{-1}\|_{L^2_\eta(-R/2,R/2)} \leq C$
uniformly in $R\geq R_*$ for all $\lambda\in B_\delta(\lambda_*)$.
\end{proposition}

Thus, while the eigenvalues of $\mathcal{L}_R$ approach the absolute spectrum $\Sigma_\mathrm{abs}$ as $R$ increases, the norm of the resolvent $(\mathcal{L}_R-\lambda)^{-1}$ will grow exponentially in $R$ for $\lambda$ to the left of the Fredholm boundary $\Sigma_\mathrm{FB}$, and in particular near the absolute spectrum $\Sigma_\mathrm{abs}$. From a numerical perspective, the eigenvalue problem is therefore ill-conditioned for $R$ large, and iterative eigenvalue solvers may not be able to locate eigenvalues accurately \cite{trefethen1,trefethen2}. Indeed, as shown in Figure~\ref{f:cd-eigs}, the eigenvalues found by MATLAB's iterative solver \texttt{eigs} for the operator $\mathcal{L}_R$ are inaccurate for all sufficiently large $R$: instead of approaching the theoretical limit $\Sigma_\mathrm{abs}$, the eigenvalues converge to $\Sigma_\mathrm{FB}$. Preconditioning with appropriate exponential weights by computing the eigenvalues of $\mathcal{L}_{R,\eta}=\rme^{\eta x}\mathcal{L}_R\rme^{-\eta x}$ recovers the predicted eigenvalues for weights $\eta\in J_0(\lambda)$.

\begin{figure}
\centering
\includegraphics[trim = 0.05cm 0.05cm 0.0cm 0.0cm,clip,width=0.33\textwidth]{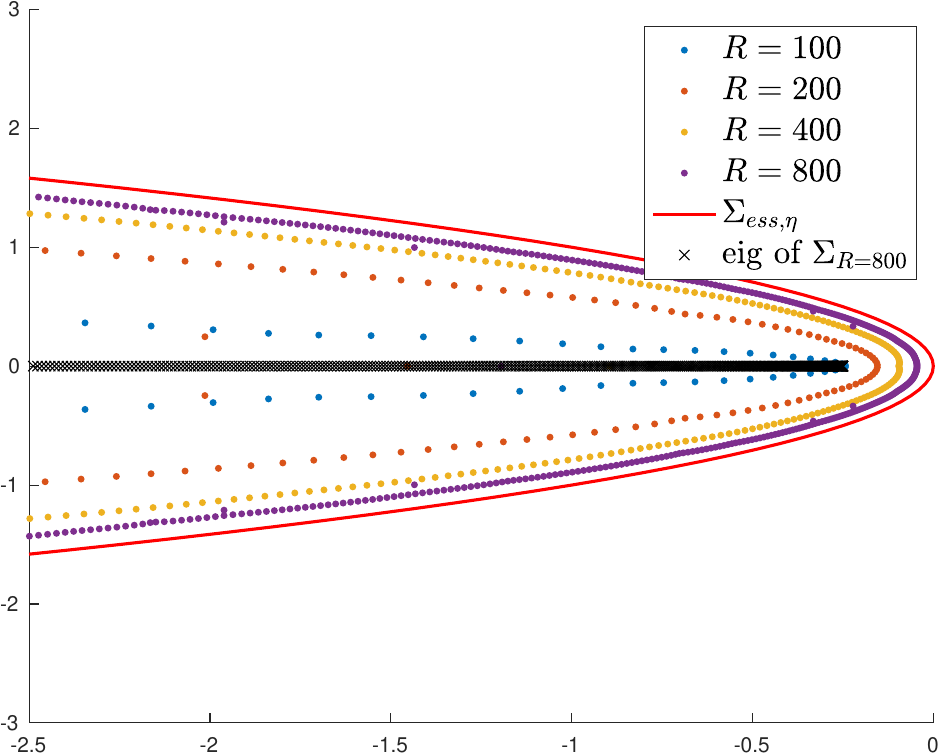}\hspace{-0.05in}
\includegraphics[trim = 0.05cm 0.05cm 0.0cm 0.0cm,clip,width=0.33\textwidth]{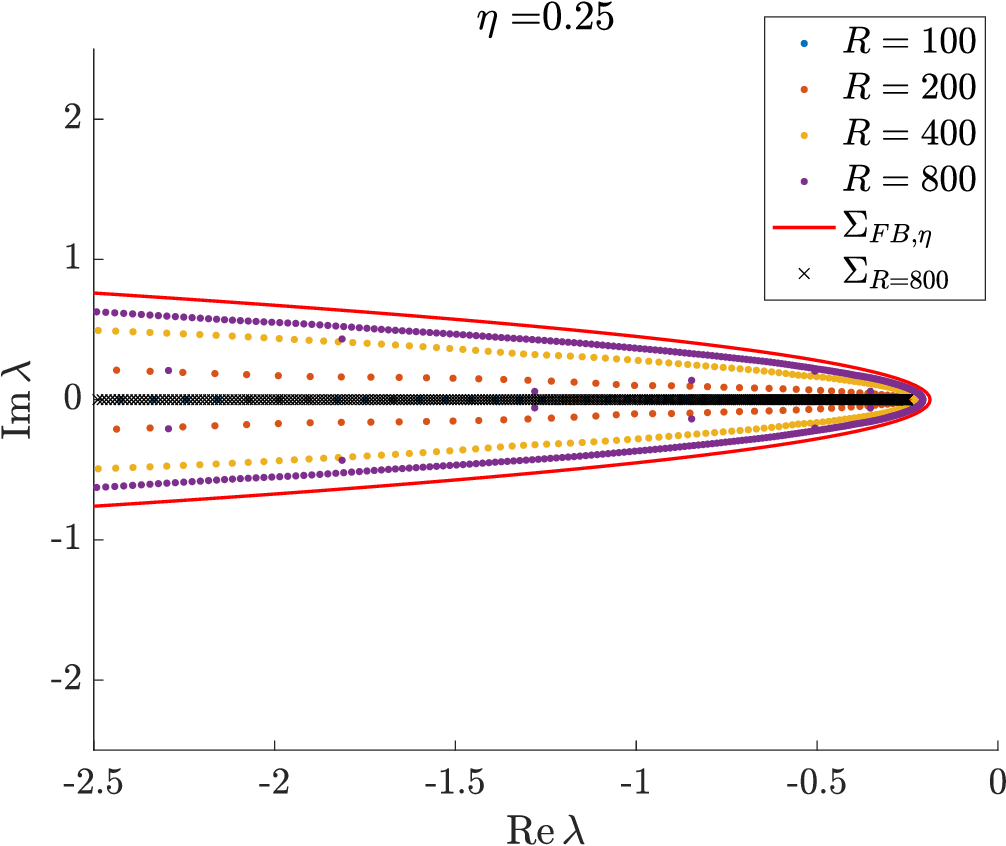}\hspace{-0.05in}
\includegraphics[trim = 0.05cm 0.05cm 0.0cm 0.0cm,clip,width=0.33\textwidth]{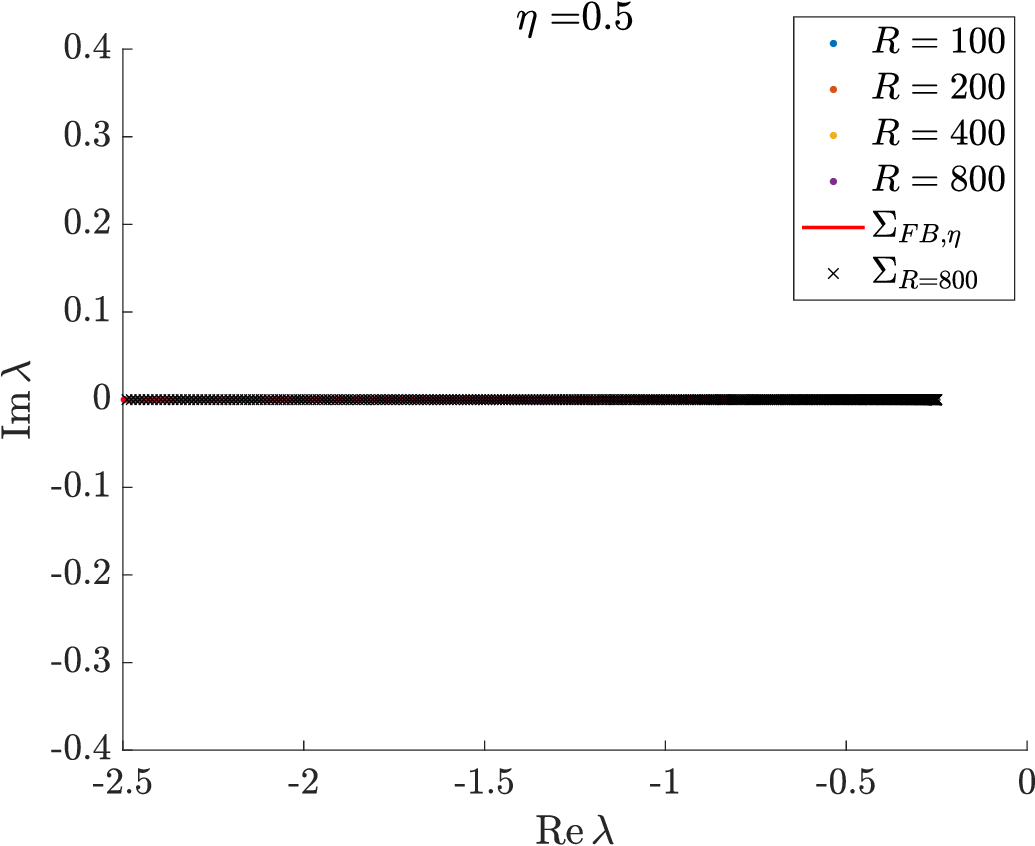}\hspace{-0.05in}
\caption{Shown are the eigenvalues of $\mathcal{L}_{R,\eta}$ with $c=1$ for different values of $R$ with $\eta=0,0.25,0.5$. Eigenvalues not visible lie in $\Sigma_\mathrm{abs}=(-\infty,-c^2/4]$. The numerical spectrum for $R=800$ with $\eta=0.5=c/2$ agrees with the theoretical spectrum within $5\times10^{-3}$ accuracy.}
\label{f:cd-eigs}
\end{figure}

\begin{figure}
\centering
\hspace{-0.5in}\includegraphics[trim = 0.5cm 0.05cm 1.5cm 0.05cm,clip,width=0.35\textwidth]{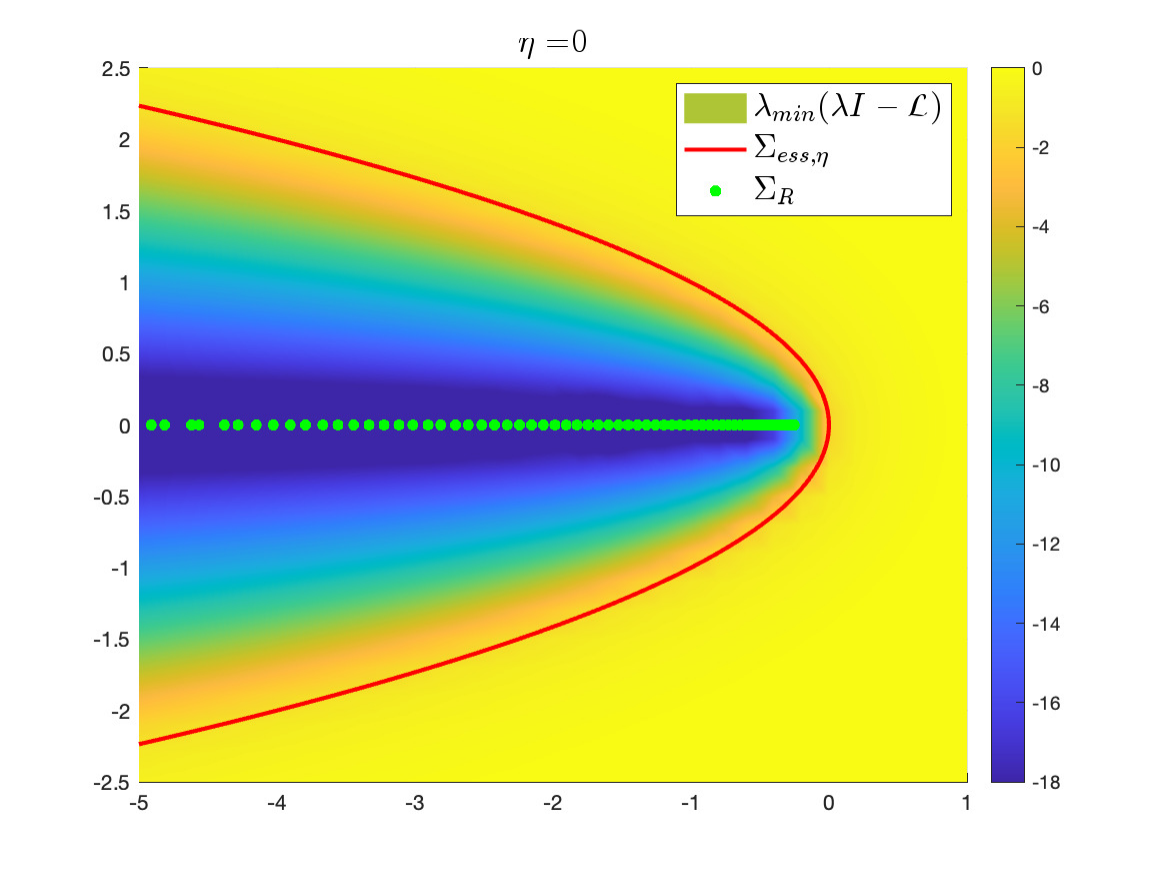}\hspace{-0.05in}
\includegraphics[trim = 0.5cm 0.05cm 1.5cm 0.05cm,clip,width=0.35\textwidth]{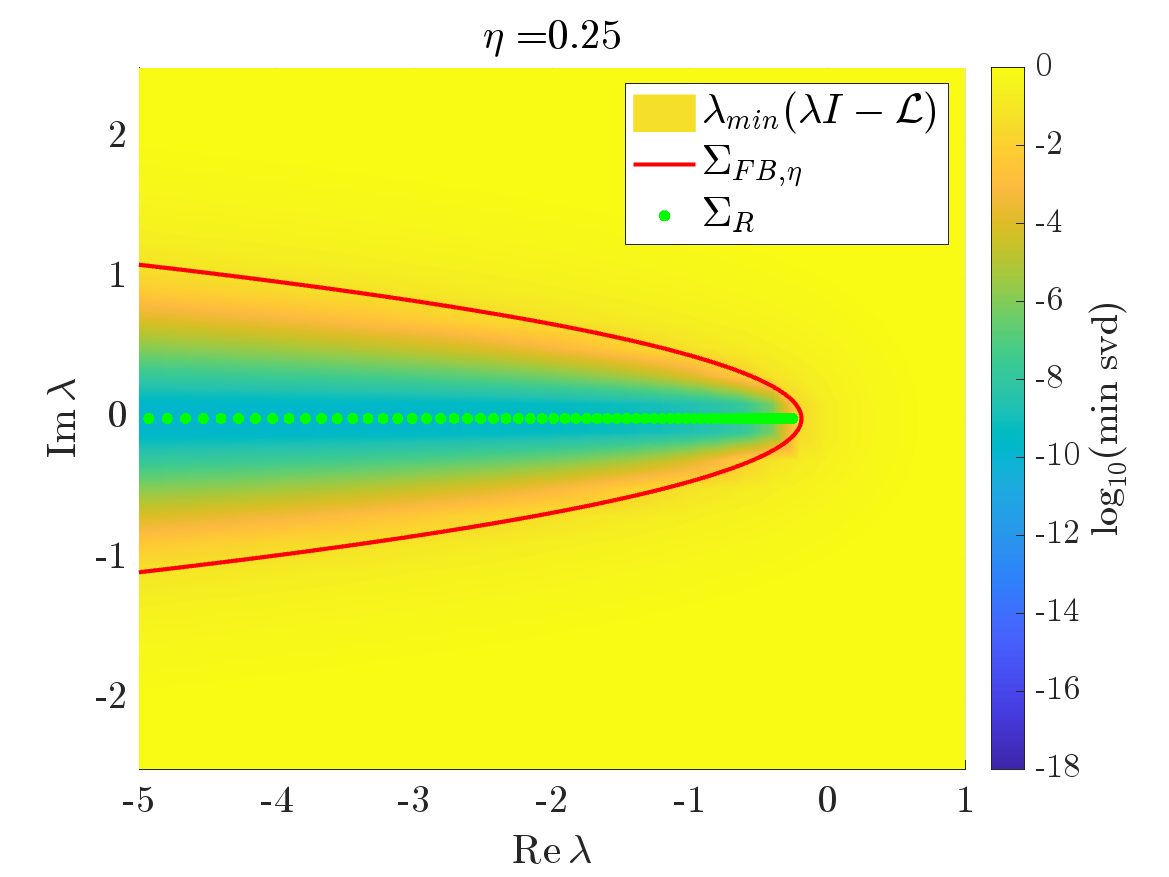}\hspace{-0.05in}
\includegraphics[trim = 0.5cm 0.05cm 0.05cm 0.05cm,clip,width=0.38\textwidth]{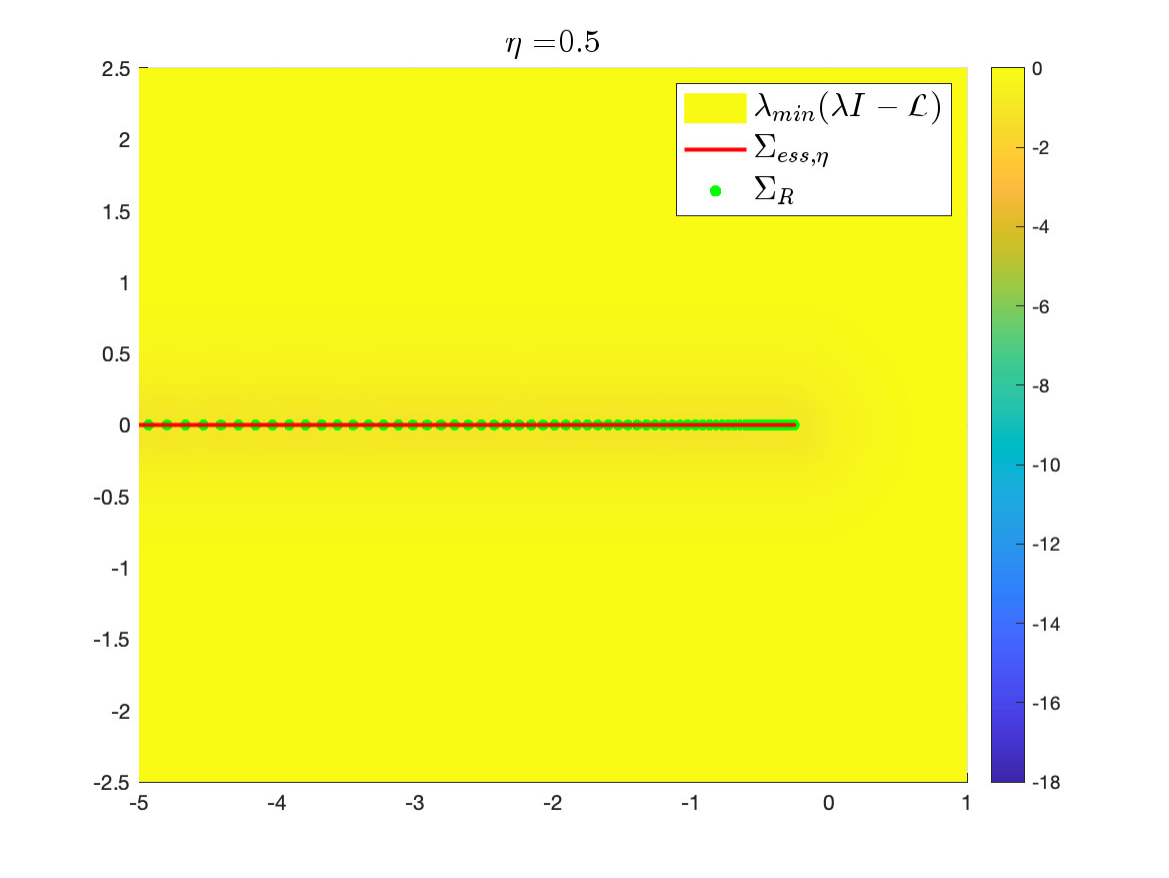}\hspace{-0.3in}
\caption{Shown are the pseudospectrum $\Lambda_\epsilon(\mathcal{L}_R)$, the Fredholm boundary $\Sigma_{\mathrm{FB},\eta}$, and the numerical eigenvalues for a range of weight values $\eta$ with $c=1$. The color scale reflects the minimal singular value of the finite-difference matrix for $\mathcal{L}_R-\lambda$ on a $\log_{10}$ scale and therefore provides the pseudospectrum contours of $\Lambda_\epsilon$. Eigenvalues were found using MATLAB's direct solver \texttt{eig}.}
\label{f:cd-ps}
\end{figure}

We illustrate the growth of the resolvent by computing the $\epsilon$-pseudospectrum, defined by $\Lambda_\epsilon(\mathcal{L}_R):=\{\lambda \colon \|(\mathcal{L}_R-\lambda)^{-1}\|_{L^2(-R/2,R/2)} \geq \epsilon^{-1}\}$ with $\epsilon>0$, numerically via the minimal singular value of the finite-difference approximation of $\mathcal{L}_R$. In Figure~\ref{f:cd-ps}, the boundaries of the pseudospectrum $\Lambda_\epsilon$ are indicated as contour lines for a range of values of $0<\epsilon\ll1$. We note that the $\epsilon$-pseudospectra are not localized around eigenvalues as would be the case for normal operators, and that to the left of $\Sigma_\mathrm{FB,\eta}$ the norm of the resolvent grows exponentially. Using positive weights $\eta>0$ shifts the Fredholm boundary and the pseudospectra $\Lambda_\epsilon(\mathcal{L}_R)$ to the left. The maximal weight value $\eta=c/2$ symmetrizes the conjugated operator $\mathcal{L}_{R,c/2}:=\partial_{xx} - c^2/4$ so that the resolvent is bounded in terms of the inverse of the distance of $\lambda$ to the spectrum of $\mathcal{L}_R$ uniformly in $R$.

%%%%%%%%%%%%%%%%%%%%%%%%%%%%%%%%%%%%%%%%%%%%%%%%%%%%%%%%%%%%%%%%%%%%%%%%%

\section{Main results}\label{s:mr}

Before stating our results, we summarize the hypotheses we shall need. We  focus first on one-dimensional wave trains, which constitute the asymptotic limits of spiral waves. Consider the reaction-diffusion system 
\begin{equation}\label{e:rd1d}
u_t = D u_{xx} + f(u), \qquad x\in\R, \quad u\in\R^N,
\end{equation}
where $D=\mathrm{diag}(d_j)>0$ is a positive, diagonal diffusion matrix and $f$ is a smooth nonlinearity. We assume that, for some non-zero temporal frequency $\omega_*$ and a certain spatial wavenumber $k_*$, there exists a traveling-wave solution $u(x,t)=u_\mathrm{wt}(k_*x-\omega_*t)$ of (\ref{e:rd1d}), where $u_\mathrm{wt}(\xi)$ is a non-constant $2\pi$-periodic function. The linearization of (\ref{e:rd1d}) about this wave train is given by
$\tilde{u}_t = D \tilde{u}_{xx} + f_u(u_\mathrm{wt}(k_*x-\omega_*t)) \tilde{u}$.
Substituting the Floquet ansatz $\tilde{u}(x,t)=\rme^{\lambda t+\nu x}u(k_*x-\omega_*t)$ into the linearization and using the notation $\phi=k_*x-\omega_*t$, we obtain the spatial eigenvalue problem
\begin{equation}\label{e:awt}
\nu \begin{pmatrix} u \\ v \end{pmatrix} = \mathcal{A}_\mathrm{wt}(\lambda) \begin{pmatrix} u \\ v \end{pmatrix}, \quad
\mathcal{A}_\mathrm{wt}(\lambda) := \begin{pmatrix} -k_*\partial_\phi & 1 \\
-D^{-1}(\omega_*\partial_\phi+f_u(u_\mathrm{wt}(\phi))-\lambda) & -k_*\partial_\phi\end{pmatrix}.
\end{equation}
We consider $\mathcal{A}_\mathrm{wt}(\lambda)$ as a closed operator on $H^{\frac12}(S^1,\C^N)\times L^2(S^1,\C^N)$ with domain $H^{\frac32}(S^1,\C^N)\times H^1(S^1,\C^N)$, where $S^1:=\R/2\pi\Z$. It was shown in \cite[Lemma~2.8]{ss-mem} that the spectrum $\spec(\mathcal{A}_\mathrm{wt}(\lambda))$ of $\mathcal{A}_\mathrm{wt}(\lambda)$ is a countable set $\{\nu_j(\lambda)\}_{j\in\Z}$ of isolated eigenvalues $\nu_j(\lambda)$ with finite multiplicity which, when ordered by increasing real part, satisfy $\Re\nu_j\to\pm\infty$ as $j\to\pm\infty$. We refer to the eigenvalues of $\mathcal{A}_\mathrm{wt}(\lambda)$ as the spatial eigenvalues. We can now formulate our hypotheses on the asymptotic wave trains.

\begin{definition}[Admissible wave trains] \label{H1}
We say that a solution $u(x,t)=u_\mathrm{wt}(k_*x-\omega_*t)$ of (\ref{e:rd1d}) is an admissible wave train if $u_\mathrm{wt}(\phi)$ is smooth and $2\pi$-periodic, the constants $k_*,\omega_*\neq0$ are nonzero, and the associated operator $\mathcal{A}_\mathrm{wt}(\lambda)$ defined in (\ref{e:awt}) satisfies the following:
\begin{compactenum}[(i)]
\item $\nu=0$ is a simple eigenvalue of $\mathcal{A}_\mathrm{wt}(0)$ with eigenfunction $(u_\mathrm{wt}^\prime,k_*u_\mathrm{wt}^{\prime\prime})$.
\item The simple eigenvalue $\nu_*(\lambda)$ of $\mathcal{A}_\mathrm{wt}(\lambda)$ with $\nu_*(0)=0$, which exists by (i), satisfies $\frac{\rmd\nu_*}{\rmd\lambda}(0)<0$.
\item We have $\spec(\mathcal{A}_\mathrm{wt}(0))\cap\rmi\R=\{0\}$.
\item For each $\lambda>0$, we have $\spec(\mathcal{A}_\mathrm{wt}(\lambda))\cap\rmi\R=\emptyset$.
\end{compactenum}
\end{definition}

Definition~\ref{H1}(i) implies that admissible wave trains arise in smooth one-parameter families that are parametrized by their wavenumber $k$ for $k$ near $k_*$ with temporal frequencies given by a smooth nonlinear dispersion relation $\omega=\omega_\mathrm{nl}(k)$. We define the group velocity of an admissible wave train by $c_\mathrm{g}:=\omega_\mathrm{nl}^\prime(k_*)$. We know from \cite[\S2.3]{ss-mem} that $c_\mathrm{g}=-[\frac{\rmd\nu_*}{\rmd\lambda}(0)]^{-1}$, and Definition~\ref{H1}(ii) therefore implies $c_\mathrm{g}>0$.

Recall that we order the spatial eigenvalues $\nu_j(\lambda)$ by increasing real part. We can choose the label of one of the spatial eigenvalues and do so by setting $\nu_{-1}(\lambda):=\nu_*(\lambda)$ for $\lambda$ near the origin: Definition~\ref{H1}(i) and~(iii) show that this choice is unambiguous. As discussed in \cite[\S2.4]{ss-mem}, this labelling can now be continued consistently, though possibly non-uniquely, to each $\lambda\in\C$. Finally, we note that Definition~\ref{H1}(iii)-(iv) implies that $\ldots\leq\Re\nu_{-1}(\lambda)<0<\Re\nu_0(\lambda)\leq\dots$
for all $\lambda>0$, and the spatial eigenvalue $\nu_{-1}(\lambda)$ crosses from left to right through the origin as $\lambda$ decreases through $0$, while we have $\Re\nu_0(0)>0$. The \emph{Fredholm boundary} $\Sigma_\mathrm{FB}$ is the set of $\lambda\in\C$ for which $A_\mathrm{wt}(\lambda)$ is not hyperbolic, and hence defines curves on which $\nu_{-1}(\lambda) \in \rmi \R$. From now on, we will fix the ordering of the spatial eigenvalues we just introduced. We can then define the spatial spectral gap
\[
J_0(\lambda) := \left(-\Re\nu_0(\lambda),-\Re\nu_{-1}(\lambda) \right)\subset \R, \qquad \lambda\in\C
\]
and note that $J_0(0)=(-\Re\nu_0(0),0)\subset\R^-$. We see in a few moments why the spatial spectral gap of wave trains is relevant for spiral waves.

\begin{definition}[Absolute spectrum]
The set $\{\lambda\in\C\colon: J_0(\lambda)=\emptyset\}$, where $\Re\nu_0(\lambda)=\Re\nu_{-1}(\lambda)$, is called the \emph{absolute spectrum} $\Sigma_\mathrm{abs}$ of the wave train $u_\mathrm{wt}$.
\end{definition}

The absolute spectrum consists of semi-algebraic curves, which generically end at branch points or cross in triple junctions \cite{ssr}. It was shown in \cite[Remark~2.4]{ss-mem} that the absolute spectrum is invariant under the vertical shifts $\lambda\mapsto\lambda+\rmi\omega_*\ell$ for $\ell\in\Z$.

Next, we consider the reaction-diffusion system
\begin{equation}\label{e:rd2d}
u_t = D \Delta u  + f(u), \qquad x\in\R^2,\quad u\in\R^N
\end{equation}
and note that, while we will often express functions using polar coordinates $(r,\varphi)$, all operators are defined in terms of the Cartesian coordinates $x\in\R^2$. We now provide a formal definition of planar Archimedean spiral waves and list the non-degeneracy conditions we need to assume for them.

\begin{definition}[Spiral waves]\label{d:sw}
We say that a solution $u(r,\varphi,t)=u_*(r,\varphi-\omega_*t)$ of (\ref{e:rd2d}) is a \emph{spiral wave} if $\omega_*>0$ and there is a wave train $u_\mathrm{wt}(k_*x-\omega_*t)$ of (\ref{e:rd1d}) with nonzero wavenumber $k_*$ and a smooth function $\theta_*(r)$ with $\theta_*^\prime(r)\to0$ as $r\to\infty$ such that $|u_*(r,\cdot-\omega_*t) - u_\mathrm{wt}(k_*r+\theta_*(r)+\cdot-\omega_*t)|_{C^1(S^1)} \to 0$ as $r\to\infty$.
\end{definition}

We linearize (\ref{e:rd2d}) in a co-rotating frame around the spiral wave to obtain the linear operator
\[
\mathcal{L}_* = D\Delta + \omega_* \partial_\varphi + f_u(u_*(r,\varphi)),
\]
which is closed and densely defined on $L^2(\R^2,\C^N)$ and whose domain contains the intersection of $H^2(\R^2,\C^N)$ and $\{u\in L^2(\R^2,\C^N)\colon \partial_\varphi u\in L^2(\R^2,\C^N)\}$; see \cite[\S3.2]{ss-mem} for further details. We will also consider the linearization $\mathcal{L}_*$ on the function spaces
$L^2_\eta(\R^2,\C^N):=\{ u\in L^2_\mathrm{loc}\colon |u|_{L^2_\eta}:=|u(x)\rme^{\eta|x|}|_{L^2}<\infty \}$.
We can now connect the spatial spectral gap $J_0(\lambda)$ with properties of the linearization $\mathcal{L}_*$: as shown in \cite[\S3.2]{ss-mem}, the operator $\mathcal{L}_*-\lambda$ is Fredholm with index zero when considered on the space $L^2_\eta(\R^2,\C^N)$ with weight $\eta\in J_0(\lambda)$. This justifies the following definition of the extended point spectrum of a planar spiral wave.

\begin{definition}[Extended point spectrum of spiral waves]
We say that $\lambda\in\mathbb{C}\setminus\Sigma_\mathrm{abs}$ is in the extended point spectrum $\Sigma_\mathrm{ext}^\mathrm{sp}$ of $\mathcal{L}_*$ if the kernel of $\mathcal{L}_*-\lambda$ is nontrivial in $L^2_\eta(\R^2,\C^N)$ for some $\eta\in J_0(\lambda)$.
\end{definition}

Since $J_0(0)=(-\Re\nu_0(0),0)\subset\R^-$ is not empty and $\mathcal{L}_*\partial_\varphi u_*=0$, we see that $\lambda=0$ lies in the extended point spectrum of the spiral wave. We will focus on transverse spiral waves which satisfy the following conditions.

\begin{definition}[Transverse spiral waves] \label{H2}
We say that a spiral wave $u_*(r,\varphi)$ is transverse if (i) its asymptotic wave train is admissible (see Definition~\ref{H1}) and (ii) for all $\eta<0$ sufficiently small the eigenvalue $\lambda=0$ of the linearization $\mathcal{L}_*$ considered as a closed operator on $L^2_\eta(\R^2,\C^N)$ is algebraically simple.
\end{definition}

This definition is slightly broader that the one given in \cite{ss-mem}, and we emphasize that all results in \cite{ss-mem} continue to hold for spiral waves that satisfy Definition~\ref{H2}.

Our next definition focuses on boundary sinks, which connect an admissible wave train at $x=-\infty$ with a boundary condition at $x=0$. Given an admissible wave train with frequency $\omega_*>0$, we seek solutions $u(x,\tau)=u_\mathrm{bs}(x,\tau)$ of the system
\begin{align}\label{e:bs}
\omega_* u_\tau & = D u_{xx} + f(u), && (x,\tau)\in\R^-\times S^1 \\ \nonumber 
0 & = a u(0,\tau) + b u_x(0,\tau), && \tau\in S^1
\end{align}
with $2\pi$-periodic boundary conditions in $\tau$, where $a,b\in\R$ with $a^2+b^2=1$. For a solution $u_\mathrm{bs}(x,\tau)$ of (\ref{e:bs}), we also consider the associated Floquet linearization
\begin{align}\label{e:bsl}
\omega_* u_\tau & = D u_{xx} + f_u(u_\mathrm{bs}(x,\tau)) u + \lambda u, && (x,\tau)\in\R^-\times S^1 \\ \nonumber
0 & = a u(0,\tau) + b u_x(0,\tau), && \tau\in S^1
\end{align}
on the space $L^2_\eta(\R^-,\C^N)$ with norm $|u|_{L^2_\eta}:=|u(x)\rme^{\eta x}|_{L^2}$. We will consider non-degenerate boundary sinks that satisfy the following conditions.

\begin{definition}[Non-degenerate boundary sinks]\label{H3}
A solution $u(x,\tau)=u_\mathrm{bs}(x,\tau)$ of (\ref{e:bs}) is called a non-degenerate boundary sink if (i) there is an admissible wave-train solution $u_\mathrm{wt}(k_*x-\omega_*t)$ of (\ref{e:rd1d}) so that $|u_\mathrm{bs}(x,\cdot)-u_\mathrm{wt}(k_*x-\cdot)|_{C^1(S^1)}\to0$ as $x\to-\infty$ and (ii) the only solution $u(x,\tau)$ of the linearization (\ref{e:bsl}) with $\lambda=0$ that satisfies $u(x,0)\in L^2_\eta(\R^-,\C^N)$ for each $\eta\in J_0(0)$ is $u=0$.
\end{definition}

We define the extended point spectrum of boundary sinks.

\begin{definition}[Extended point spectrum of boundary sinks]
We say that $\lambda\in\mathbb{C}\setminus\Sigma_\mathrm{abs}$ is in the extended point spectrum $\Sigma_\mathrm{ext}^\mathrm{bs}$ of a non-degenerate boundary sink $u_\mathrm{bs}(x,\tau)$ if (\ref{e:bsl}) has a nontrivial solution $u(x,\tau)$ with $u(x,0)\in L^2_\eta(\R^-,\C^N)$ for some $\eta\in J_0(\lambda)$.
\end{definition}

It was shown in \cite[Theorem~3.19]{ss-mem} that, given numbers $a,b\in\R$ with $a^2+b^2=1$, a transverse spiral wave persists under truncation as a solution $u_R$ of the boundary-value problem
\begin{align}\label{e:disk}
0 = D\Delta u + \omega u_\varphi + f(u) \mbox{ for } |x|<R \quad\mbox{and}\quad
0 = a u + b \frac{\partial u}{\partial n} \mbox{ for } |x|=R
\end{align}
for all large $R\gg1$ provided (\ref{e:bs}) admits a non-degenerate boundary sink belonging to the admissible asymptotic wave train of the planar spiral wave (and we refer to \S\ref{s:p2} for a comparison of temporal frequencies and profiles of the planar spiral $u_*$, the boundary sink $u_\mathrm{bs}$, and the truncated spiral $u_R$). We now turn to our main results. We define
\begin{equation}\label{e:lr}
\mathcal{L}_R := D\Delta + \omega_R \partial_\varphi + f_u(u_R(r,\varphi))
\end{equation}
in Cartesian coordinates as a densely defined operator on $L^2(B_R(0),\C^N)$ with domain
$\mathrm{D}(\mathcal{L}_R):=\{u\in H^2(B_R(0),\C^N)\colon \left(au+b\frac{\partial u}{\partial n}\right)|_{|x|=R}=0\}$.
We also set
$\Sigma_\infty := \Sigma_\mathrm{abs} \cup \Sigma_\mathrm{ext}^\mathrm{sp} \cup \Sigma_\mathrm{ext}^\mathrm{bs}$
and note that it follows from the proof of \cite[Theorem~3.26]{ss-mem} that the spectrum $\spec(\mathcal{L}_R)$ of $\mathcal{L}_R$ on $L^2(B_R(0))$ with domain $\mathrm{D}(\mathcal{L}_R)$ lies in the $\epsilon$-neighborhood $U_\epsilon(\Sigma_\infty)$ of $\Sigma_\infty$ inside each compact subset of $\C$ for all $R\gg1$. In fact, if the extended point spectra of the spiral wave and the boundary sink do not intersect, and the absolute spectrum satisfies additional simplicity and non-resonance conditions, then \cite[Theorem~3.26]{ss-mem} shows that $\spec(\mathcal{L}_R)\to\Sigma_\infty$ in the symmetric Hausdorff distance as $R\to\infty$ uniformly on each fixed compact subset of $\C$. In other words, it is expected that infinitely many eigenvalues converge to $\Sigma_\mathrm{abs}$ as $R\to\infty.$

Analogously to the case of one-dimensional patterns considered in \cite{ss-trunc}, we expect that the norm of the resolvent $(\mathcal{L}_R-\lambda)^{-1}$ grows exponentially in $R$ for each $\lambda\in\C\setminus U_\epsilon(\Sigma_\infty)$ for which $\mathcal{L}_*-\lambda$ has a non-zero Fredholm index (that is, where $\Re\nu_{-1}(\lambda)>0$ or $\Re\nu_0(\lambda)<0$). Our first result affirms this expectation. 

\begin{theorem}\label{p:rg}
Assume that $u_*(r,\varphi)$ is a transverse spiral wave of (\ref{e:rd2d}) with admissible asymptotic wave train $u_\mathrm{wt}(k_*x-\omega_*t)$ and that $u_\mathrm{bs}(x,\tau)$ is a non-degenerate boundary sink of (\ref{e:bs}) belonging to the wave train $u_\mathrm{wt}$. For each $\lambda_*\notin\Sigma_\infty$, there are constants $C_*,R_*,\delta_*,\eta_*>0 $ so that the following is true:
\begin{compactenum}[(i)]
\item Assume $0\notin\overline{J_0(\lambda_*)}$: If $\Re\nu_{-1}(\lambda_*)>0$ with $\Re\nu_{-2}(\lambda_*)<\Re\nu_{-1}(\lambda_*)$, or else $\Re\nu_0(\lambda_*)<0$ with $\Re\nu_0(\lambda_*)<\Re\nu_1(\lambda_*)$, then $\|(\mathcal{L}_R-\lambda)^{-1}\|_{L^2(B_R(0))}\geq C_*\rme^{\eta_*R}$ uniformly in $R\geq R_*$ and $\lambda\in B_{\delta_*}(\lambda_*)$.
\item If $0\in J_0(\lambda_*)$, then $\|(\mathcal{L}_R-\lambda)^{-1}\|_{L^2(B_R(0))}\leq C_*$ uniformly in $R\geq R_*$ and $\lambda\in B_{\delta_*}(\lambda_*)$.
\end{compactenum}
\end{theorem}

Next, consider $\mathcal{L}_R$ with domain $\mathcal{Y}^1_\eta:=\{ u\in H^2_\eta(B_R(0))\colon(au+b\frac{\partial u}{\partial n})|_{|x|=R} = 0\}$ on the space $L^2_\eta(B_R(0))$ where
$|u|_{L^2_\eta(B_R(0))}^2 = \int_{|x|\leq R} |u(x)\rme^{\eta|x|}|^2\,\rmd x$, then the resolvent is bounded uniformly in $R$ on $L^2_\eta(B_R(0))$ for appropriate rates $\eta$.

\begin{theorem}\label{t:1}
Assume that $u_*(r,\varphi)$ is a transverse spiral wave of (\ref{e:rd2d}) with admissible asymptotic wave train $u_\mathrm{wt}(k_*x-\omega_*t)$ and that $u_\mathrm{bs}(x,\tau)$ is a non-degenerate boundary sink of (\ref{e:bs}) belonging to the wave train $u_\mathrm{wt}$, then there exists a $C^0$-function $\eta:\C\setminus\Sigma_\mathrm{abs}\to\R$ with $\eta(\lambda)\in J_0(\lambda)$ for each $\lambda$ so that the following is true. For each compact subset $\Omega$ of $\C$ and each $\epsilon>0$, there are numbers $C_*,R_*>0$ so that 
$\|(\mathcal{L}_R-\lambda)^{-1}\|_{L(L^2_{\eta(\lambda)}(B_R(0)))} \leq C_*$
for all $\lambda\in\Omega\setminus U_\epsilon(\Sigma_\infty)$ and $R>R_*$, where $\mathcal{L}_R$ is the operator defined in \eqref{e:lr} posed on $L^2_{\eta(\lambda)}(B_R(0))$ with domain $\mathcal{Y}^1_{\eta(\lambda)}$.
\end{theorem}

%%%%%%%%%%%%%%%%%%%%%%%%%%%%%%%%%%%%%%%%%%%%%%%%%%%%%%%%%%%%%%%%%%%%%%%%%

\section{Proof of main results}\label{s:proof}

We prove Theorem~\ref{t:1} in \S\ref{s:p0}-\ref{s:p6} and Theorem~\ref{p:rg} in \S\ref{s:p7}.

\subsection{Spatial dynamics}\label{s:p0}

Recall the operator
$\mathcal{L}_R = D\Delta + \omega_R \partial_\varphi + f_u(u_R(r,\varphi))$
on $L^2_\eta(B_R(0))$ with domain 
$\mathcal{Y}^1_\eta$,
where $u_R(r,\varphi)$ denotes the truncated spiral-wave solution of \eqref{e:disk} for $\omega=\omega_R$, whose existence is guaranteed by our assumptions. Choose a compact subset $\Omega\subset\C$ and a constant $\epsilon$ with $0<\epsilon\ll1$, and define the compact set
\begin{equation}\label{e:ceps}
\Lambda_\epsilon := \Omega\setminus U_\epsilon(\Sigma_\infty).
\end{equation}
Pick a continuous function $\eta:\Lambda_\epsilon\to\R$ with $\eta(\lambda)\in J_0(\lambda)$ for all $\lambda\in\Lambda_\epsilon$. In this setting, we want to find constants $C_*,R_*>0$ so that for each $R\geq R_*$, $\lambda\in\Lambda_\epsilon$, and $h\in L^2_{\eta(\lambda)}(B_R(0))$ the equation $(\mathcal{L}_R-\lambda) w = h$ has a unique solution $w\in\mathcal{Y}^1_{\eta(\lambda)}$ and we have $|w|_{L^2_{\eta(\lambda)}}\leq C_*|h|_{L^2_{\eta(\lambda)}}$.
We will reformulate this problem as follows. Given any $\eta\in J_0(\lambda)$ and $h\in L^2_\eta$, we write $g=\rme^{\eta|x|}h$ so that $g\in L^2$ with $|g|_{L^2}=|h|_{L^2_\eta}$. Writing $u=\rme^{\eta|x|}w$, we see that the problem described above is equivalent to finding constants $C_*,R_*$ so that
\begin{equation}\label{e:evp1}
(\rme^{\eta|x|} \mathcal{L}_R \rme^{-\eta|x|}-\lambda) u = g
\end{equation}
has a unique solution $u\in\mathcal{X}^1_\eta:=\{ u\in H^2(B_R(0))\colon \left((a-b\eta)u + bu_r\right)|_{|x|=R} = 0\}$ with $|u|_{L^2}\leq C_*|g|_{L^2}$ for each $R\geq R_*$. Our strategy for proving this claim for \eqref{e:evp1} is to write this equation in polar coordinates as the first-order differential equation
\begin{align}\label{e:evp3}
\begin{pmatrix} u_r \\ v_r \end{pmatrix} =
\begin{pmatrix} \eta & 1 \\
- \frac{\partial_{\varphi\varphi}}{r^2} - D^{-1}[\omega_R\partial_\varphi + f^\prime(u_R) - \lambda] & \eta - \frac{1}{r}
\end{pmatrix}
\begin{pmatrix} u \\ v \end{pmatrix}
+ \begin{pmatrix} 0 \\ D^{-1} g(r,\varphi) \end{pmatrix}
\end{align}
in the spatial evolution variable $r\in(0,R)$, where $(u,v)(r,\cdot)$ lies for each fixed $r$ in the Banach space $X:=H^1(S^1,\C^N)\times L^2(S^1,\C^N)$. The boundary conditions for $u\in\mathcal{X}^1_\eta$ at $|x|=R$ translate into the $\eta$-independent boundary conditions
\[
(u,v)(R,\cdot) \in E^\mathrm{bc} := \left\{ (u,v)\in X\colon au+bv=0 \right\}
\]
for solutions $(u,v)(r,\cdot)$ of \eqref{e:evp3} at $r=R$. We will now discuss \eqref{e:evp3} in different regions of $B_R(0)$. We will rely heavily on the results established in \cite{ss-mem} to which we refer for details and proofs of the facts we quote below.

\subsection{Archimedean coordinates}\label{s:p2}

In \cite{ss-mem}, the planar and truncated spiral waves were constructed as smooth profiles in Archimedean coordinates. As in \cite{ss-mem}, we therefore define $u_*^\mathrm{a}(r,\vartheta):=u_*(r,\vartheta-k_*r-\theta_*(r))$, $u_R^\mathrm{a}(r,\vartheta):=u_R(r,\vartheta-k_Rr-\theta_R(r))$, and $u_\mathrm{bs}^\mathrm{a}(\rho,\vartheta):=u_\mathrm{bs}(\rho,k_*\rho-\vartheta)$ to denote the truncated spiral wave, the planar spiral wave, and the boundary sink in Archimedean coordinates, where $k_R$ and $\theta_R(r)$ are the wave number and phase correction functions associated with the truncated spiral wave. It was shown in \cite[Theorem~3.19 and \S9.2]{ss-mem} that
\begin{eqnarray}\label{e:a2}
\lefteqn{\sup_{0\leq r\leq R-\kappa^{-1}\log R} |u_R^\mathrm{a}(r,\vartheta)-u_*^\mathrm{a}(r,\vartheta)| \rme^{\kappa(R-\kappa^{-1}\log R-r)}} \\ \nonumber &&
+ \sup_{-\kappa^{-1}\log R\leq\rho\leq0} |u_R^\mathrm{a}(R+\rho,\vartheta)-u_\mathrm{bs}^\mathrm{a}(\rho,\vartheta)| \leq \frac{C}{R^\gamma}
\end{eqnarray}
and $|\omega_*-\omega_R| + |k_*-k_R| + \sup_{0\leq r\leq R} r |\theta_*(r)-\theta_R(r)| \leq C\rme^{-\gamma R}$.
Furthermore, we have $u_*^\mathrm{a}(r,\cdot)\to u_\mathrm{wt}(\cdot)$ as $r\to\infty$ and $u_\mathrm{bs}^\mathrm{a}(\rho,\cdot)\to u_\mathrm{wt}(\cdot)$ as $\rho\to-\infty$.

Below, we will also need the spectral projections of the linearization $\mathcal{A}_\mathrm{wt}(\lambda)$ defined in \eqref{e:awt} on $Y:=H^{\frac12}(S^1,\C^N)\times L^2(S^1,\C^N)$. We showed in \S\ref{s:mr} that $\spec(\mathcal{A}_\mathrm{wt}(\lambda))\cap-\eta+\rmi\R=\emptyset$ for each $\eta\in J_0(\lambda)$. We can therefore define the complementary spectral projections $P^\mathrm{s,u}_\mathrm{wt}(\lambda)\in L(Y)$ of $\mathcal{A}_\mathrm{wt}(\lambda)$ associated with the spectral sets $\{\nu\in\spec(\mathcal{A}_\mathrm{wt}(\lambda)):\Re\nu<-\eta\}$ and $\{\nu\in\spec(\mathcal{A}_\mathrm{wt}(\lambda)):\Re\nu>-\eta\}$. Note that these projections do not depend on $\eta$ as long as $\eta$ lies in $J_0(\lambda)$.

\subsection{Far-field region}\label{s:p3}

We first consider the region $r\in[R_1,R-\kappa^{-1}\log R]$ for an appropriate $R_1>0$ and all $R\gg1$. Since the truncated spiral wave $u_R(r,\varphi)$ is close to the planar spiral wave $u_*(r,\varphi)$ in this region by \eqref{e:a2}, we first discuss the linearization around the planar spiral wave. Using the Archimedean coordinates $\vartheta=k_*r+\theta_*(r)+\varphi$ instead of $\varphi$ in \eqref{e:evp3} and setting $g=0$, we arrive at the homogeneous spatial dynamical system
\begin{align}\label{e:sw1}
\mathbf{u}_r = \mathcal{A}^\eta_*(r,\lambda) \mathbf{u}, \qquad \mathbf{u} = (u,v)
\end{align}
with
\begin{equation}\label{e:sw2}
\mathcal{A}^\eta_*(r,\lambda) = \begin{pmatrix}
\eta - (k_*+\theta_*^\prime(r))\partial_\vartheta & 1 \\
- \frac{\partial_{\vartheta\vartheta}}{r^2} - D^{-1}[\omega_*\partial_\vartheta + f_u(u_*^\mathrm{a}(r,\vartheta)) - \lambda]
& \eta - (k_*+\theta_*^\prime(r))\partial_\vartheta - \frac{1}{r} \end{pmatrix}.
\end{equation}
For each fixed $r>0$, the operator $\mathcal{A}^\eta_*(r,\lambda)$ is closed on the Banach space $X:=H^1(S^1,\C^N)\times L^2(S^1,\C^N)$ with dense domain $X^1:=H^2(S^1,\C^N)\times H^1(S^1,\C^N)$. We equip $X$ with the $r$-dependent norm
$|\mathbf{u}(r)|_{X_r}^2 := \frac{1}{r^2} |u|_{H^1}^2 + |u|_{H^{1/2}}^2 + |v|_{L^2}^2$
and write $X_r$ whenever the $r$-dependence of the norm is important. In \cite[Lemma~5.4]{ss-mem}, we constructed linear isomorphisms $\mathcal{I}(r):X_r\to Y$ with $\|\mathcal{I}(r)\|_{L(X_r,Y)}\leq C$ uniformly in $r\geq1$ that allowed us to transfer the spectral projections $P^\mathrm{s,u}_\mathrm{wt}(\lambda)$ defined in \S\ref{s:p2} on the space $Y$ to the $r$-dependent projections $P^\mathrm{s,u}_\mathrm{wt}(r;\lambda):=\mathcal{I}(r)P^\mathrm{s,u}_\mathrm{wt}(\lambda)\mathcal{I}(r)\in L(X_r)$ on $X_r$.
We can now discuss the solvability of the equation
\begin{equation}\label{e:asd}
\mathbf{u}_r = \mathcal{A}(r)\mathbf{u}, \qquad \mathbf{u}\in X_r,
\end{equation}
where $\mathcal{A}(r)$ is of the form \eqref{e:sw2} with $(u_*,\omega_*)$ possibly replaced by other profiles and temporal frequencies. The key notion is exponential dichotomies:

\begin{definition}[Exponential dichotomy]
We say that \eqref{e:asd} has an exponential dichotomy with constant $K$ and rate $\alpha>0$ on an interval $J\subset\R^+$ if there are linear operators $\Phi^\mathrm{s}(r,\rho)$, defined for $r\geq\rho$ in $J$, and $\Phi^\mathrm{u}(r,\rho)$, defined for $r\leq\rho$ in $J$, so that the following is true. For all $r\geq\rho$ in $J$, we have
$\|\Phi^\mathrm{s}(r,\rho)\|_{L(X_\rho,X_r)} + \|\Phi^\mathrm{u}(\rho,r)\|_{L(X_r,X_\rho)} \leq K\rme^{-\alpha|r-\rho|}$.
For each $\mathbf{u}_0\in X_\rho$, the functions $\mathbf{u}(r)=\Phi^\mathrm{s}(r,\rho)\mathbf{u}_0$ and $\mathbf{u}(r)=\Phi^\mathrm{u}(r,\rho)\mathbf{u}_0$ satisfy \eqref{e:asd} for $r\geq\rho$ and $r\leq\rho$, respectively, in $J$. The operators $P^\mathrm{s}(r):=\Phi^\mathrm{s}(r,r)$ and $P^\mathrm{u}(r):=\Phi^\mathrm{u}(r,r)$ are complementary projections on $X_r$, which are strongly continuous in $r$, and we have $\Rg(\Phi^\mathrm{s}(r,\rho))=\Rg(P^\mathrm{s}(r))$ for $r\geq\rho$ in $J$ and $\Rg(\Phi^\mathrm{u}(r,\rho))=\Rg(P^\mathrm{u}(r))$ for $r\leq\rho$ in $J$.
\end{definition}

As shown in \cite{pss}, exponential dichotomies persist under small perturbations.

\begin{lemma}[Robustness]\label{l:edr}
Assume that \eqref{e:asd} has an exponential dichotomy with constant $K$ and rate $\alpha>0$ on the interval $J\subset\R^+$. For each $\delta_0>0$ and $\alpha_0\in(0,\alpha)$, there are constants $K_0,\delta_1>0$ so that the perturbed system $\mathbf{u}_r = \mathcal{A}(r) \mathbf{u} + \mathcal{B}(r) \mathbf{u}$
with $\|\mathcal{B}(r)\|_{\mathrm{L}(X_r)}\leq\delta_1$ for $r\in J$ has an exponential dichotomy on $J$ with constant $K_0$ and rate $\alpha_0$, and the associated projections are $\delta_0$-close to the projections for (\ref{e:asd}) uniformly in $r\in J$.
\end{lemma}

It was shown in \cite[\S5.2 and \S5.5]{ss-mem} that for each $\lambda\in\Lambda_\epsilon$ and $\eta\in J_0(\lambda)$ there is an $R_1>0$ so that \eqref{e:sw1} has an exponential dichotomy with constant $K$ and rate $\alpha>0$ on $[R_1,\infty)$, and we denote the associated projections by $P^\mathrm{s,u}_*(r;\lambda,\eta)$. It follows from \cite[Proposition~5.5]{ss-mem} that $\|P^\mathrm{s}_\mathrm{wt}(r;\lambda)-P^\mathrm{s}_*(r;\lambda,\eta)\|_{L(X_r)}\to0$ as $r\to\infty$. Any positive number $\alpha$ that satisfies $\eta\pm\alpha\in J_0(\lambda)$ can be chosen as the rate of the exponential dichotomy.
Next, we consider the homogeneous system
\begin{align}\label{e:tsw1}
\mathbf{u}_r = \mathcal{A}^\eta_R(r,\lambda) \mathbf{u}
\end{align}
associated with the truncated spiral wave $u_R$ on $X_r$, where
\begin{equation}\label{e:tsw2}
\mathcal{A}^\eta_R(r,\lambda) = \begin{pmatrix}
\eta - (k_R+\theta_R^\prime(r))\partial_\vartheta & 1 \\
- \frac{\partial_{\vartheta\vartheta}}{r^2} - D^{-1}[\omega_R\partial_\vartheta + f_u(u_R^\mathrm{a}(r)) - \lambda]
& \eta - (k_R+\theta_R^\prime(r))\partial_\vartheta - \frac{1}{r} \end{pmatrix}.
\end{equation}
The estimate \eqref{e:a2} shows that $u_R^\mathrm{a}$ is $1/R^\gamma$-close to the planar spiral wave $u_*^\mathrm{a}$ in the region we consider here, and we conclude that for each $\lambda\in\Lambda_\epsilon$ and $\eta\in J_0(\lambda)$ there are constants $\delta,C>0$ so that
\[\textstyle
\sup_{r\in[R_1,R-\kappa^{-1}\log R]} \left\|\mathcal{A}^\eta_*(r,\lambda) - \mathcal{A}^{\tilde{\eta}}_R(r,\tilde{\lambda})\right\|_{L(X_r)} \leq C \left( \frac{1}{R^\gamma} + |\lambda-\tilde{\lambda}| + |\eta-\tilde{\eta}| \right)
\]
uniformly in $R$ for all $\lambda\in U_\delta(\lambda)$ and $\eta\in U_\delta(\eta)$, where $\gamma$ has been defined in \eqref{e:a2}. Extending $\mathcal{A}^\eta_R(r,\lambda)$ from $[R_1,R-\kappa^{-1}\log R]$ to $[R_1,\infty)$ by freezing its coefficients at their value at $r=R-\kappa^{-1}\log R$ and applying Lemma~\ref{l:edr} gives the following result.

\begin{lemma}[Far-field dichotomies]\label{l:ff}
For each $\delta_0>0$, $\lambda_0\in\Lambda_\epsilon$, and $\eta_0\in J_0(\lambda_0)$ there exist positive constants $\alpha,\delta,K,R_1,R_2$ so that the following is true. Equation \eqref{e:tsw1} has an exponential dichotomy with constant $K$ and rate $\alpha$ on $J=[R_1,R-\kappa^{-1}\log R]$ uniformly in $\lambda\in U_\delta(\lambda_0)$, $\eta\in U_\delta(\eta_0)$, and $R\geq R_2$, and the associated projections $P^\mathrm{s}_R(r;\lambda,\eta)$ satisfy
$\sup_{r\in[R_1,R-\kappa^{-1}\log R]} \|P^\mathrm{s}_R(r;\lambda,\eta) - P^\mathrm{s}_*(r;\lambda_0,\eta_0)\|_{X_r} \leq \delta_0$.
\end{lemma}

\subsection{Boundary-layer region}\label{s:p4}

We consider the region $r\in[R-\kappa^{-1}\log R,R]$. To facilitate comparison with the boundary sink, we use the independent variable $\rho=r-R$ instead of $r$. The linearization about the boundary sink $u_\mathrm{bs}^\mathrm{a}(\rho,\vartheta)$ in Archimedean coordinates is given by
\[\textstyle
u_\rho = \eta u - k_*\partial_\vartheta u + v, \quad
v_\rho = \eta v - k_*\partial_\vartheta v - \frac{v}{\rho+R}
- D^{-1}[\omega_R\partial_\vartheta + f_u(u_\mathrm{bs}^\mathrm{a}(\rho,\vartheta)) - \lambda] u
\]
with $\rho\in\R^-$. We know from \cite[Lemma~9.1]{ss-mem} that this equation has an exponential dichotomy on $\R^-$ and that the associated projections $P^\mathrm{s}_\mathrm{bs}(\rho;\lambda,\eta)$ converge to $P^\mathrm{s}_\mathrm{wt}(\lambda)$ as $\rho\to-\infty$. Since $\lambda\notin\Sigma_\mathrm{ext}^\mathrm{bs}$ and $\eta\in J_0(\lambda)$, we also know that $\Rg(P^\mathrm{u}_\mathrm{bs}(0;\lambda,\eta))\oplus E^\mathrm{bc}=X$, and the expressions in \cite[(3.20)]{pss} show that we can modify the exponential dichotomy of the boundary sink on $\R^-$ so that $\Rg(P^\mathrm{s}_\mathrm{bs}(0;\lambda,\eta))=E^\mathrm{bc}$. Next, we reformulate the linearization \eqref{e:tsw1} around $u_R^\mathrm{a}(\rho+R,\vartheta)$ using the coordinate $\rho=r-R$ to arrive at
\begin{align}\label{e:tbl1}
u_\rho = & \eta u - [k_R + \theta_R^\prime(\rho+R)]\partial_\vartheta u + v \\ \nonumber
v_\rho = & \eta v - [k_R + \theta_R^\prime(\rho+R)]\partial_\vartheta v \textstyle
- \frac{v}{\rho+R} - \frac{\partial_{\vartheta\vartheta} u}{(\rho+R)^2} \\ \nonumber & \textstyle
- D^{-1}[\omega_R\partial_\vartheta + f_u(u_R^\mathrm{a}(\rho+R,\vartheta)) - \lambda] u
\end{align}
with $\rho\in[-\kappa^{-1}\log R,0]$. The estimate \eqref{e:a2} shows that $u_R^\mathrm{a}(\rho+R,\cdot)$ is $1/R^\gamma$-close to the boundary sink $u_\mathrm{bs}^\mathrm{a}(\rho,\cdot)$. Using Lemma~\ref{l:edr} and the results in \cite[\S5.2 and \S9.2]{ss-mem} for the system \eqref{e:tbl1} with coefficients frozen at their values at $\rho=-\kappa^{-1}\log R$, we have the following result.

\begin{lemma}[Boundary-layer dichotomies]\label{l:bl}
Given $\delta_0>0$, $\lambda_0\in\Lambda_\epsilon$, and $\eta_0\in J_0(\lambda_0)$ there exist constants $\alpha,\delta,K,R_2>0$ so that the following is true. Equation \eqref{e:tbl1} has an exponential dichotomy with constant $K$ and rate $\alpha$ on $[-\kappa^{-1}\log R,0]$ uniformly in $\lambda\in U_\delta(\lambda_0)$, $\eta\in U_\delta(\eta_0)$, and $R\geq R_2$, and the associated projections $\tilde{P}^\mathrm{s}_R(\rho;\lambda,\eta)$ satisfy $\Rg(\tilde{P}^\mathrm{s}_R(0;\lambda,\eta))=E^\mathrm{bc}$ and
$\sup_{\rho\in[-\kappa^{-1}\log R,0]} \|\tilde{P}^\mathrm{s}_R(\rho;\lambda,\eta) - P^\mathrm{s}_\mathrm{bs}(\rho;\lambda_0,\eta_0)\|_{X_r} \leq \delta_0$.
\end{lemma}

\subsection{Matching far-field and boundary-layer regions}\label{s:p5}

First, we combine the results we obtained in \S\ref{s:p3} and \S\ref{s:p4} to conclude the existence of exponential dichotomies of the linearization
\begin{align}\label{e:tsw0}
\mathbf{u}_r = \mathcal{A}^\eta_R(r,\lambda) \mathbf{u}
\end{align}
associated with the truncated spiral wave $u_R$ on the interval $[R_1,R]$ for all $R\gg1$, where $\mathcal{A}^\eta_R(r,\lambda)$ has been defined in \eqref{e:tsw2}. Choose $\Lambda_\epsilon$ as in \eqref{e:ceps} and pick continuous functions $\eta_\pm(\lambda)$ so that $[\eta_-(\lambda),\eta_+(\lambda)]\in J_0(\lambda)$ for all $\lambda\in\Lambda_\epsilon$. We then define the compact set $\mathcal{C}_\epsilon:=\{(\lambda,\eta):\lambda\in\Lambda_\epsilon, \eta\in[\eta_-(\lambda),\eta_+(\lambda)]\}$.

\begin{proposition}\label{p:1}
Assume that the assumptions of Theorem~\ref{t:1} are met and choose $\mathcal{C}_\epsilon$ as above. For each $\delta_0>0$, there exist positive constants $\alpha,K,R_1,R_2$ so that the following is true. Equation \eqref{e:tsw0} has an exponential dichotomy $\Phi^\mathrm{s,u}_R(r,\rho;\lambda,\eta)$ with constant $K$ and rate $\alpha$ on $J=[R_1,R]$ uniformly in $(\lambda,\eta)\in\mathcal{C}_\epsilon$ and $R\geq R_2$, and the associated projections $P^\mathrm{s,u}_R(r;\lambda,\eta)$ satisfy $\Rg(P^\mathrm{s}_R(R;\lambda,\eta))=E^\mathrm{bc}$ and
\begin{eqnarray*}
\lefteqn{\sup_{r\in[R_1,R-\kappa^{-1}\log R]} \|P^\mathrm{s}_R(r;\lambda,\eta) - P^\mathrm{s}_*(r;\lambda,\eta)\|_{X_r}} \\ &&
+ \sup_{r\in[R-\kappa^{-1}\log R,R]} \|P^\mathrm{s}_R(r;\lambda,\eta) - P^\mathrm{s}_\mathrm{bs}(r-R;\lambda,\eta)\|_{X_r} \leq \delta_0.
\end{eqnarray*}
\end{proposition}

\begin{proof}
We proved the existence of exponential dichotomies for \eqref{e:tsw0} in Lemmas~\ref{l:ff} and~\ref{l:bl} separately on $[R_1,R-\kappa^{-1}\log R]$ and  $[R-\kappa^{-1}\log R,R]$. Since the associated projections evaluated at $r=R-\kappa^{-1}\log R$ are arbitrarily close to the spectral projections of the wave-train projections, we can use \cite[(3.20)]{pss} to redefine the projections and exponential dichotomies of \eqref{e:tsw0} so that they are continuous at $r=R-\kappa^{-1}\log R$ and therefore give dichotomies on $[R_1,R]$. This process does not change the rate $\alpha$ and replaces the constant $K$ by $K(1+2K)^2$. Next, the results in Lemmas~\ref{l:ff} and~\ref{l:bl}, and therefore the extension we just discussed, are locally uniform, and we can use compactness of $\mathcal{C}_\epsilon$ to prove that, given $\delta_0>0$, the radii $R_1,R_2$, the constant $K$, and the rate $\alpha$ can be chosen uniformly in $(\lambda,\eta)\in\mathcal{C}_\epsilon$.
\end{proof}

Next, given $g\in L^2(B_R(0))$, we need to solve \eqref{e:evp3} and establish uniform estimates for the solution. We switch to Archimedean coordinates, define $\mathbf{g}(r):=(0,D^{-1}g(r,\cdot))^*$, and rewrite \eqref{e:evp3} as $\mathbf{u}_r=\mathcal{A}^\eta_R(r,\lambda)\mathbf{u}+\mathbf{g}(r)$. Using \cite[\S6.2]{ss-mem} and Proposition~\ref{p:1}, we see that the function
\begin{align}\label{e:soln-ff}
\mathbf{u}_+(r) & = \Phi_R^\mathrm{s}(r,R_1;\lambda,\eta) \mathbf{a}^\mathrm{s}_+
+ \Phi_R^\mathrm{u}(r,R;\lambda,\eta) \mathbf{a}^\mathrm{u}_+
+ \int_{R_1}^r \Phi_R^\mathrm{s}(r,\rho;\lambda,\eta) \mathbf{g}(\rho)\,\rmd\rho \\ \nonumber &
+ \int_R^r \Phi_R^\mathrm{u}(r,\rho;\lambda,\eta) \mathbf{g}(\rho)\,\rmd\rho
\end{align}
is a solution with 
\begin{equation}\label{e:bdd-ff}\textstyle
\sup_{r\in[R_1,R]} |\mathbf{u}_+(r)|_{X_r} \leq K\left( |\mathbf{a}^\mathrm{s}_+|_{X_{R_1}} + |\mathbf{a}^\mathrm{u}_+|_{X_R} + \frac{2}{\alpha} |g|_{L^2(B_R(0))} \right)
\end{equation}
for arbitrary $\mathbf{a}^\mathrm{s}_+\in\Rg(P_R^\mathrm{s}(R_1;\lambda,\eta))$ and $\mathbf{a}^\mathrm{u}_+\in\Rg(P_R^\mathrm{u}(R;\lambda,\eta))$. 

\subsection{Core region}\label{s:p1}

It remains to analyse the region $r\in[0,R_1]$ with $R_1$ as in Proposition~\ref{p:1}. This region was investigated in \cite[\S5.1, \S5.3, and \S5.5]{ss-mem}, and we therefore only summarize the results proved there. The equation $\mathbf{u}_r=\mathcal{A}^\eta_R(r,\lambda)\mathbf{u}+\mathbf{g}(r)$ has exponential dichotomies $\widehat{P}_R^\mathrm{s,u}(r;\lambda,\eta)$ on $[0,R_1]$ and has for each $\mathbf{b}^\mathrm{u}_-\in\Rg(\widehat{P}_R^\mathrm{u}(R_1;\lambda,\eta))$ a unique bounded solution $\mathbf{u}_-(r)$ with
\begin{eqnarray}
\sup_{r\in[0,R_1]} |\mathbf{u}_-(r)|_X & \leq & K_1\left( |\mathbf{b}^\mathrm{u}_-|_{X_{R_1}} + 2|g|_{L^2(B_R(0))} \right)
\label{e:bdd-c}\textstyle \\ \label{e:soln-c}
\mathbf{u}_-(R_1) & = & \textstyle \mathbf{b}^\mathrm{u}_- + \int_0^{R_1} \widehat{\Phi}_{R}^\mathrm{s}(R_1,\rho;\lambda,\eta) \mathbf{g}(\rho)\,\rmd\rho 
\end{eqnarray}
uniformly in $(\lambda,\eta)$. This completes the analysis of \eqref{e:evp3} for $r\in[0,R_1]$.

\subsection{Uniform resolvent estimates}\label{s:p6}

Equations \eqref{e:soln-ff} and \eqref{e:soln-c} provide solutions $\mathbf{u}_+(r)$ and $\mathbf{u}_-(r)$ of $\mathbf{u}_r=\mathcal{A}^\eta_R(r,\lambda)\mathbf{u}+\mathbf{g}(r)$ on $[R_1,R]$ and $[0,R_1]$, respectively. It remains to solve the matching conditions $\mathbf{u}_+(R_1)=\mathbf{u}_-(R_1)$ and the boundary conditions $\mathbf{u}_+(R)\in E^\mathrm{bc}$, which are given by
\begin{eqnarray}\label{e:match1}
0 & = &
\mathbf{a}^\mathrm{s}_+ + \Phi_R^\mathrm{u}(R_1,R;\lambda,\eta) \mathbf{a}^\mathrm{u}_+ + \int_R^{R_1} \Phi_R^\mathrm{u}(R_1,\rho;\lambda,\eta) \mathbf{g}(\rho)\,\rmd\rho
\\ \nonumber &&
- \mathbf{b}^\mathrm{u}_- - \int_0^{R_1} \widehat{\Phi}_{R}^\mathrm{s}(R_1,\rho;\lambda,\eta) \mathbf{g}(\rho)\,\rmd\rho \\ \label{e:match2}
\mathbf{u}_+(R) & = & \mathbf{a}^\mathrm{u}_+ +\Phi_R^\mathrm{s}(R,R_1;\lambda,\eta) \mathbf{a}^\mathrm{s}_+
+ \int_{R_1}^R \Phi_R^\mathrm{s}(R,\rho;\lambda,\eta) \mathbf{g}(\rho)\,\rmd\rho \in E^\mathrm{bc}.
\end{eqnarray}
Since $\lambda\notin\Sigma_\mathrm{ext}^\mathrm{sp}$, it follows from \cite[Proposition~6.1 and \S5.5]{ss-mem} that the map
\[
\iota(\lambda,\eta)\colon \quad
\Rg(P^\mathrm{s}_R(R_1;\lambda,\eta))\times\Rg(\widehat{P}^\mathrm{u}_R(R_1;\lambda,\eta)) \longrightarrow X_{R_1}, \quad
(\mathbf{a}^\mathrm{s}_+,\mathbf{b}^\mathrm{u}_-) \longmapsto \mathbf{a}^\mathrm{s}_+ - \mathbf{b}^\mathrm{u}_-
\]
is invertible with inverse that is bounded uniformly in $(\lambda,\eta)$. Furthermore, we know from Proposition~\ref{p:1} that $\Rg(P^\mathrm{u}_R(R;\lambda,\eta))\oplus E^\mathrm{bc}=X_R$, and we also have
\[\textstyle
|\Phi_R^\mathrm{u}(R_1,R;\lambda,\eta) \mathbf{a}^\mathrm{u}_+|_{X_{R_1}}
+ |\Phi_R^\mathrm{s}(R,R_1;\lambda,\eta) \mathbf{a}^\mathrm{s}_+|_{X_R} \leq 
K\rme^{-\alpha(R-R_1)} \left(|\mathbf{a}^\mathrm{u}_+|_{X_{R_1}} + |\mathbf{a}^\mathrm{s}_+|_{X_R}\right)
\]
Thus, we can solve \eqref{e:match1}-\eqref{e:match2} uniquely for $(\mathbf{a}^\mathrm{s}_+,\mathbf{a}^\mathrm{u}_+,\mathbf{b}^\mathrm{u}_-)$ as a linear function of $\mathrm{g}$ with bound
\[
|\mathbf{a}^\mathrm{s}_+|_{X_{R_1}} + |\mathbf{a}^\mathrm{u}_+|_{X_R} + |\mathbf{b}^\mathrm{u}_-|_{X_{R_1}} \leq C_0 |g|_{L^2(B_R(0))},
\]
where $C_0=C_0(K,K_1,\alpha)$ does not depend on $R$ or on $(\lambda,\eta)\in\mathcal{C}_\epsilon$. Substituting these bounds into the estimates \eqref{e:bdd-c} and \eqref{e:bdd-ff} shows that the first component $u(r,\varphi)$ of $\mathbf{u}(r)$ satisfies $|u|_{L^2(B_R(0))} \leq C_1 |g|_{L^2(B_R(0))}$, where $C_1$ does not depend on $R$ or on $(\lambda,\eta)\in\mathcal{C}_\epsilon$. Finally, the arguments in \cite[\S6.2]{ss-mem} show that $u\in H^2(B_R(0))$. This completes the proof of Theorem~\ref{t:1}.

\subsection{Proof of Theorem~\ref{p:rg}}\label{s:p7}

Theorem~\ref{p:rg}(ii) follows directly from Theorem~\ref{t:1} since we can apply this theorem with $\eta=0$ which lies in $J_0(\lambda_*)$ by assumption. It therefore remains to prove statement~(i), which claims that $\|(\mathcal{L}_R-\lambda)^{-1}\|$ grows exponentially in $R$. Since we assume that $0\notin\overline{J_0(\lambda_*)}=[-\Re\nu_0(\lambda_*),-\Re\nu_{-1}(\lambda_*)]$, we have either $\Re\nu_{-1}(\lambda_*)>0$ or else $\Re\nu_0(\lambda_*)<0$. We focus on the case $\Re\nu_{-1}(\lambda_*)>0$ and will comment later on the second case, which can be tackled analogously. Throughout the proof, we will fix $\lambda_*$ and $\eta_*\in J_0(\lambda_*)$ so that $\eta_*<0$ and omit the superscripts and the dependence on $(\lambda,\eta)$ in the remainder of this section, since variations in $(\lambda,\eta)$ can be included as in the previous sections.

We denote by $\mathbf{u}_{-1}=(u_{-1},v_{-1})$ the eigenvector of $\mathcal{A}^\eta_\mathrm{wt}:=\mathcal{A}_\mathrm{wt}+\eta\mathds{1}$ belonging to the simple eigenvalue $\tilde{\nu}_{-1}:=\nu_{-1}+\eta$ and by $\mathbf{u}^\mathrm{ad}_{-1}$ the corresponding eigenvector of the adjoint operator. It follows from \cite[\S4.3]{ss-mem} that $v_{-1}\neq0$. For later use, we also define $P^\mathrm{c}_\mathrm{wt}\mathbf{v}:=\langle\mathbf{u}^\mathrm{ad}_{-1},\mathbf{v}\rangle \mathbf{u}_{-1}$. Our strategy for establishing Theorem~\ref{p:rg} is to prove that the norm $|w|_{L^2(B_R(0))}$ of the solution $w$ of
\begin{equation}\label{e:rg0}
(\mathcal{L}_R-\lambda) w = h, \qquad
h(r,\varphi) := \chi_{[R_2,R_2-d]}(r) v_{-1}(\varphi)
\end{equation}
grows exponentially in $R$, where $\chi_I(r)$ denotes the indicator function of the interval $I\subset\R$ and $R_2,d$ are $R$-independent constants that we will choose later. We will rely on the results in \S\ref{s:p3} for the linear system (\ref{e:tsw1})
\begin{align}\label{e:rg1}
\mathbf{u}_r = \mathcal{A}^\eta_R(r,\lambda) \mathbf{u}
\end{align}
associated with the truncated spiral wave $u_R$, posed on the exponentially weighted spaces and extended to $[R_1,\infty)$ by freezing its coefficients at their value for $r=R-\kappa^{-1}\log R$. Lemma~\ref{l:ff} shows that \eqref{e:rg1} has exponential dichotomies $\Phi^\mathrm{s,u}_R(r,s)$ with constant $K$ and rate $\alpha>0$ on $[R_1,\infty)$ and that the associated projections $P^\mathrm{s}_R(r)$ satisfy
\[
\sup_{r\in[R_1,R-\kappa^{-1}\log R]} \|P^\mathrm{s}_R(r) - P^\mathrm{s}_*(r)\|_{X_r} \leq \frac{C}{R^\gamma}.
\]
Our first result provides asymptotic expansions of bounded solutions to \eqref{e:rg1}. Recall that we assumed that $\Re\nu_{-2}(\lambda)<\Re\nu_{-1}(\lambda)$.

\begin{lemma}\label{l:rg1}
There are positive constants $a_0,\beta,C_0>0$, constants $b_0\in\R$ and $R_0\geq R_1$, a real-valued function $a(r,s)$, and a projection $P^\mathrm{c}_R(r)$ so that $\|P^\mathrm{c}_R(r)-P^\mathrm{c}_\mathrm{wt}\|\leq C_0(\frac{1}{r}+\frac{1}{R^\gamma})$, $|a(r,s)|\leq a_0$ for all $r\geq s$, and
\begin{equation}\label{e:rg2}
\left|\Phi^\mathrm{s}_R(r,s)\mathbf{v}_0 - \left(\frac{r}{s}\right)^{b_0} \rme^{\tilde{\nu}_{-1}(r-s)} \rme^{a(r,s)} P^\mathrm{c}_R(s)\mathbf{v}_0 \right|_{X_r} \leq C_0 \rme^{(\tilde{\nu}_{-1}-\beta)(r-s)} |\mathbf{v}_0|_{X_s}
\end{equation}
uniformly in $R_0\leq s\leq r\leq R-3\kappa^{-1}\log R$ for each $\mathbf{v}_0\in X_s$.
\end{lemma}

\begin{proof}
It was shown in \cite[Equation~(8.7), Proposition~10.4, and Step~4 in \S11.3]{ss-mem} that solutions $\mathbf{v}(r)$ of \eqref{e:rg1} in the center-stable directions can be written in the form $\mathbf{v}(r)=\mathbf{v}^\mathrm{c}(r)+\Phi^\mathrm{ss}_R(r,s)\mathbf{v}(s)$, where $\|\Phi^\mathrm{ss}_R(r,s)\|\leq C_0\rme^{(\tilde{\nu}_{-1}-\beta)(r-s)}$ for each fixed $\beta\in(0,\Re\nu_{-1}-\Re\nu_{-2})$, and $\mathbf{v}^\mathrm{c}(r)=P^\mathrm{c}_R(r)\mathbf{v}(r)$ satisfies the scalar linear ODE
\[
\mathbf{v}^\mathrm{c}_r = \left[\tilde{\nu}_{-1} + \frac{b_0}{r} + \rmO\left(\frac{1}{r^2} + \rme^{-\kappa(R-\kappa^{-1}\log R-r)}\right) \right] \mathbf{v}^\mathrm{c}.
\]
Integrating this equation gives the expression \eqref{e:rg2}, where $a(r,s)$ is given by
\[
\int_s^r \rmO\left(\frac{1}{\rho^2} + \rme^{-\kappa(R-\kappa^{-1}\log R-\rho)}\right) \,\rmd\rho \leq C_0\left(\frac{1}{R_0} + \frac{1}{R}\right)
\]
for $R_0\leq s\leq r\leq R-3\kappa^{-1}\log R$. This completes the proof of the lemma.
\end{proof}

Next, consider \eqref{e:rg0} written in exponentially weighted spaces as
\begin{equation}\label{e:rg3}
\mathbf{u}_r = \mathcal{A}^\eta_R(r,\lambda) \mathbf{u} + \mathbf{g}(r), \quad
\mathbf{g}(r) = \begin{pmatrix} 0 \\ D^{-1} g(r,\cdot) \end{pmatrix}, \quad
g(r,\cdot) := \rme^{\eta r} h(r,\cdot).
\end{equation}
Note that $w(r):=\rme^{-\eta r}P_1\mathbf{u}(r)$ is then the solution of \eqref{e:rg0} and that we have
\[
|g|_{L^2(B_R(0))} = \rme^{\eta R_2} \frac{\sqrt{(1-2\eta R_2)(\rme^{-2\eta d}-1)}}{2\eta} |v_{-1}|_{L^2(S^1)}.
\]
In \S\ref{s:p6}, we constructed solutions of \eqref{e:rg3} via a variation-of-constants formula on $[R_1,R]$ after combining the exponential dichotomies we had previously constructed separately in the far field $[R_1,R-\kappa^{-1}\log R]$ and the boundary-layer region $[R-\kappa^{-1}\log R,R]$. Here, we will instead use the far-field dichotomies on $[R_1,R-\kappa^{-1}\log R]$ and introduce a second matching step at $r=R-\kappa^{-1}\log R$ with the solution in the boundary-layer region. Proceeding in the same way as in \S\ref{s:p6}, we find that the solution $\mathbf{u}(r)$ of \eqref{e:rg3} is of the form
\begin{eqnarray*}
\mathbf{u}(r) & = & \Phi_R^\mathrm{s}(r,R_1) \mathbf{a}^\mathrm{s}
+ \Phi_R^\mathrm{u}(r,R-\kappa^{-1}\log R) \mathbf{a}^\mathrm{u}
+ \int_{R_1}^r \Phi_R^\mathrm{s}(r,\rho) \mathbf{g}(\rho)\,\rmd\rho \\ &&
+ \int_{R-\kappa^{-1}\log R}^r \Phi_R^\mathrm{u}(r,\rho) \mathbf{g}(\rho)\,\rmd\rho
\end{eqnarray*}
for $r\in[R_1,R-\kappa^{-1}\log R]$, where $\mathbf{a}^\mathrm{s}$ and $\mathbf{a}^\mathrm{u}$ arise from the matching conditions and satisfy
\begin{equation}\label{e:rg5}
|\mathbf{a}^\mathrm{s}|_{X_{R_1}} + |\mathbf{a}^\mathrm{u}|_{X_{R-\kappa^{-1}\log R}} \leq C_0|g|_{L^2(B_R(0))}
\leq C_0 \rme^{\eta R_2} \sqrt{R_2d} |v_{-1}|_{L^2(S^1)}.
\end{equation}
Since the stable projections $P^\mathrm{s}_R(r)$ are uniformly close to the wave-train projections for $r\geq R_1$, we conclude that there is a constant $c_0>0$ so that
\begin{equation}\label{e:rg4}
|\mathbf{u}(r)|_{X_r} \geq c_0 \left| \Phi_R^\mathrm{s}(r,R_1) \mathbf{a}^\mathrm{s} + \int_{R_1}^r \Phi_R^\mathrm{s}(r,\rho) \mathbf{g}(\rho)\,\rmd\rho \right|_{X_r}
\end{equation}
uniformly in $r\geq R_1$.

\begin{lemma}\label{l:rg2}
Choose $\epsilon$ so that $0<\epsilon<\min\{\beta,\nu_{-1}\}$, then there are constants $c_1,d>0$ and $R_3\geq R_2\geq R_1$ so that the solution of \eqref{e:rg3} satisfies $|\mathbf{u}(r)|_{X_r}\geq c_1\rme^{(\tilde{\nu}_{-1}-\epsilon)r}$ uniformly in $r\in[R_3,R-3\kappa^{-1}\log R]$.
\end{lemma}

\begin{proof}
We focus on \eqref{e:rg4} and define $\mathbf{g}_0:=(0,D^{-1}v_{-1})$. For $r\geq R_2$, equation \eqref{e:rg3} and regularity of $\mathbf{g}_0$ show that
\begin{align*}
\int_{R_1}^r \Phi_R^\mathrm{s}(r,\rho) \mathbf{g}(\rho)\,\rmd\rho
& = \Phi_R^\mathrm{s}(r,R_2) \int_{R_2-d}^{R_2} \Phi_R^\mathrm{s}(R_2,\rho) \rme^{\eta\rho} \mathbf{g}_0\,\rmd\rho \\
& = \Phi_R^\mathrm{s}(r,R_2) \mathbf{g}_0 d (1+\rmO(d)) \rme^{\eta R_2},
\end{align*}
where the $\rmO(d)$ term is bounded uniformly in $R_2$. Hence, for $r\geq R_2$, we have
\begin{align*}
\Phi_R^\mathrm{s}(r,R_1) \mathbf{a}^\mathrm{s} + \int_{R_1}^r \Phi_R^\mathrm{s}(r,\rho) \mathbf{g}(\rho)\,\rmd\rho
 & = \Phi_R^\mathrm{s}(r,R_2) \left[ \mathbf{g}_0 d (1+\rmO(d)) \rme^{\eta R_2} + \Phi_R^\mathrm{s}(R_2,R_1) \mathbf{a}^\mathrm{s} \right] \\ 
 & =: \rme^{\eta R_2} \Phi_R^\mathrm{s}(r,R_2) \mathbf{g}_1 
\end{align*}
and \eqref{e:rg5} shows that 
\[
\left|\mathbf{g}_1 - \mathbf{g}_0 d \right| \leq C_1 \left(d^2 + \sqrt{\frac{R_2}{d}} \rme^{-\alpha(R_2-R_1)}\right) |\mathbf{g}_0|,
\]
where $C_1$ does not depend on $R_2$ and $d$. Using Lemma~\ref{l:rg1}, we conclude that
\begin{equation}\label{e:rg7}
\left| \Phi_R^\mathrm{s}(r,R_2) \mathbf{g}_1 -
\left(\frac{r}{R_2}\right)^{b_0} \rme^{\tilde{\nu}_{-1}(r-R_2)} \rme^{a(r,R_2)} P^\mathrm{c}_R(R_2) \mathbf{g}_1 \right|_{X_r} \leq C_0 \rme^{(\tilde{\nu}_{-1}-\beta)(r-R_2)} |\mathbf{g}_1|_{X_{R_2}}.
\end{equation}
Note that \cite[\S4.3]{ss-mem} and algebraic simplicity of the spatial eigenvalue $\nu_{-1}$ imply that
\[
|P^\mathrm{c}_R(R_2) \mathbf{g}_0|_{X_{R_2}} \geq |P^\mathrm{c}_\mathrm{wt} \mathbf{g}_0|_{X_{R_2}} - \frac{C_0}{R_2} |\mathbf{g}_0|_{X_{R_2}} \geq 1 - \frac{C_0|D^{-1}|}{R_2} \geq \frac{1}{2}
\]
for all sufficiently large $R_2$. Hence, we see that
\begin{align*}
|P^\mathrm{c}_R(R_2) \mathbf{g}_1|_{X_{R_2}}
& \geq |P^\mathrm{c}_R(R_2) d \mathbf{g}_0|_{X_{R_2}} - |P^\mathrm{c}_R(R_2) (\mathbf{g}_1-d\mathrm{g}_0)|_{X_{R_2}} \\
& \geq \frac{d}{2} - C_1 \left(d^2 + \sqrt{\frac{R_2}{d}} \rme^{-\alpha(R_2-R_1)}\right) |D^{-1}| |v_{-1}|_{L^2{S^1}}
\geq \frac{d}{4}
\end{align*}
after first choosing $d$ small enough and then $R_2$ large enough. Using these estimates together with \eqref{e:rg7}, we see that \eqref{e:rg4} becomes
\begin{align*}
|\mathbf{u}(r)|_{X_r} & \geq c_0 \left| \Phi_R^\mathrm{s}(r,R_1) \mathbf{a}^\mathrm{s} + \int_{R_1}^r \Phi_R^\mathrm{s}(r,\rho) \mathbf{g}(\rho)\,\rmd\rho \right|_{X_r}
= c_0 \rme^{\eta R_2} \left| \Phi_R^\mathrm{s}(r,R_2) \mathbf{g}_1 \right|_{X_r} \\
& \geq \frac{c_0 d}{4} \rme^{\eta R_2} \left(\frac{r}{R_2}\right)^{b_0} \rme^{\tilde{\nu}_{-1}(r-R_2)} \rme^{a(r,R_2)}
- c_0 C_0 |d| \rme^{\eta R_2} \rme^{(\tilde{\nu}_{-1}-\beta)(r-R_2)}.
\end{align*}
Choose $\epsilon>0$ so small that $\epsilon<\beta$ and $\nu_{-1}-\epsilon>0$. Since $c_0,d>0$ and $|a(r,R_2)|\leq a_0$ uniformly in $r$, we see that there are constants $c_1=c_1(R_2,d)>0$ and $R_3\geq R_2$ so that
$|\mathbf{u}(r)|_{X_r} \geq c_1 \rme^{(\tilde{\nu}_{-1}-\epsilon)r}$
uniformly in $r\in[R_3,R-3\kappa^{-1}\log R]$, which completes the proof of the lemma.
\end{proof}

Finally, $\mathbf{v}(r)=\rme^{-\eta r}\mathbf{u}(r)$ is the corresponding solution in the unweighted space. We have
$|\mathbf{v}(r)|_{X_r} = \rme^{-\eta r} |\mathbf{u}(r)|_{X_r} \geq c_1 \rme^{(\tilde{\nu}_{-1}-\eta-\epsilon)r} = c_1 \rme^{(\nu_{-1}-\epsilon)r}$
and by construction we have $\nu_{-1}-\epsilon>0$, which guarantees exponential growth of $\mathbf{v}(r)$ in $r$ on the interval $[R_3,R-3\kappa^{-1}\log R]$ and therefore for its first component $w(r)$ which satisfies \eqref{e:rg0}.

This completes the proof of Theorem~\ref{p:rg} for the case $\Re\nu_{-1}(\lambda)>0$. The case where $\Re\nu_0(\lambda)<0$ can be treated similarly by focusing on the unstable directions in backward time, instead of the stable directions in forward time. We omit the details as they are similar to the case studied above.

%%%%%%%%%%%%%%%%%%%%%%%%%%%%%%%%%%%%%%%%%%%%%%%%%%%%%%%%%%%%%%%%%%%%%%%%%

\section{Algorithm and numerical validation}\label{s:num}
The resolvent bounds of Theorem~\ref{t:1} provide the basis for a numerical algorithm to accurately and efficiently compute the eigenvalues of a spiral wave posed on a bounded domain. In this section, we first describe the algorithmic framework and then apply it to the Barkley model. 

\subsection{Exponential weights as preconditioners}

We seek to numerically approximate the spectra of the operator $\mathcal{L}_R$ posted on the bounded disk $B_R(0)$. For the numerical computations, the Laplacian is defined in polar coordinates $(r,\psi)$ and the relevant operator is
$\mathcal{L}_R \mathbf{v} = D \Delta_{r,\psi}  \mathbf{v} + \omega \partial_{\psi} \mathbf{v} + f_\mathbf{u}\left(\mathbf{u}_*(r,\psi)\right) \mathbf{v}$, 
which acts on functions in $\{\mathbf{v} \in H^2(B_R(0))\colon \mathbf{v}_r(R,\cdot) = 0\}$.

Posing the operator $\mathcal{L}_R$ in the exponentially weighted space $L^2_{\eta}(B_R(0))$ is equivalent to seeking eigenfunctions of the form $\mathbf{v}(r,\psi) = \rme^{-\eta r} \mathbf{w}(r,\psi)$. Thus, we instead consider the linear operator
$\mathcal{L}^{\eta}_{R} \mathbf{w}:=\rme^{\eta r} \mathcal{L}^{\eta}_{R} \rme^{-\eta r} \mathbf{w}=\mathcal{L}_{R}\mathbf{w}+D[ \eta^2 - \frac{\eta}{r}- 2\eta \partial_r] \mathbf{w}$ on the space $\{\mathbf{w} \in H^2(B_R(0))\colon \mathbf{w}_r(R,\cdot) = \eta \mathbf{w}(R,\cdot)\}$. Note that the operator $\mathcal{L}_R^{\eta}$ becomes $\mathcal{L}_R$ for $\eta = 0$. Based on Theorem~\ref{t:1}, we choose the exponential weight $\eta(\lambda)$ in the interval $J_0(\lambda)$, which is determined by the spectrum of $A_\mathrm{wt}(\lambda)$.

We note that Theorem~\ref{p:rg}(ii) provides uniform bounds on the resolvent for each $\lambda\in\mathbb{C}$ for which $0\in J_0(\lambda)$, that is, informally, for all $\lambda$ to the right of $\Sigma_\mathrm{FB}$ in the unweighted space. Hence, iterative eigenvalue solvers should work as expected to identify eigenvalues in these regions. Thus, the use of the weighted operator $\mathcal{L}^{\eta}_R$ is particularly useful for $\lambda \in \mathbb{C}$ for which $0 \notin \overline{J_0(\lambda)}$.

\textbf{Numerical methods.} Computing the spectra of spiral waves involves first solving for the spiral-wave patterns and subsequently computing the eigenvalues of the linearized operator. The spiral wave $\mathbf{u}_*(r,\psi)$ and far-field periodic wave-train solutions are computed numerically via root-finding methods following established methods: we review these methods briefly and refer to \cite{BW,dodson2019} for additional details. All computations are done in MATLAB, and the code is available on GitHub \cite{code_dodson_goh_sandstede}. Periodic wave trains are computed on a one-dimensional $2\pi$-periodic domain using a pseudospectral method with 128 grid points. For the spiral-wave computations, the bounded disk domain becomes a rectangle in polar coordinates, which we discretize with $N_r$ radial grid points and $N_{\theta}$ angular grid points. Derivatives are approximated using fourth-order centered finite differences in the radial direction and Fourier differentiation matrices in the angular coordinate. The radial grid spacing is fixed at $h_r = 0.05$ with $N_r = R/h_r+1$ radial grid points.

For the eigenvalue computations, the linear operator $\mathcal{L}_R^{\eta}$ is formed using differentiation matrices on a grid with a single grid point at the origin. At the origin, $\partial_r\mathbf{w}= 0$ and the Laplacian is computed with a five-point stencil. Boundary conditions applied on the outer radius are enforced using second-order centered finite-difference schemes coupled with the ghost point method. Numerically approximating eigenvalues of $\mathcal{L}^{\eta}_R$ is equivalent to finding the eigenvalues of a sparse square matrix with dimension $\left[N_{\theta} (N_r -1) + 1 \right]$ for each component of the equation. Unless stated otherwise, the 400 eigenvalues with the smallest absolute value are computed using the sparse eigenvalue solver $\texttt{eigs}$ with the `smallestabs' option.

The absolute spectrum and Fredholm boundaries are computed using the asymptotic periodic wave trains via the continuation algorithms described in \cite{ssr}. The $\epsilon$-pseudospectrum of $\mathcal{L}_R^{\eta}$ is found via the minimum singular value of the shifted operator $\mathcal{L}_R^{\eta} - \lambda$ for a grid of $\lambda \in \mathbb{C}$. Singular values were computed with the \texttt{svds} function. Condition numbers of the same shifted operator $\mathcal{L}_R^{\eta} - \lambda$ are computed using the \texttt{condest} function. Spatial eigenvalues $\nu(\lambda)$ are approximated numerically by computing eigenvalues of the operator $A_\mathrm{wt}(\lambda)$ defined in (\ref{e:awt}), where  derivatives were approximated via a Fourier spectral method with 128 grid points.

\begin{figure}
\centering
\includegraphics[width=\linewidth]{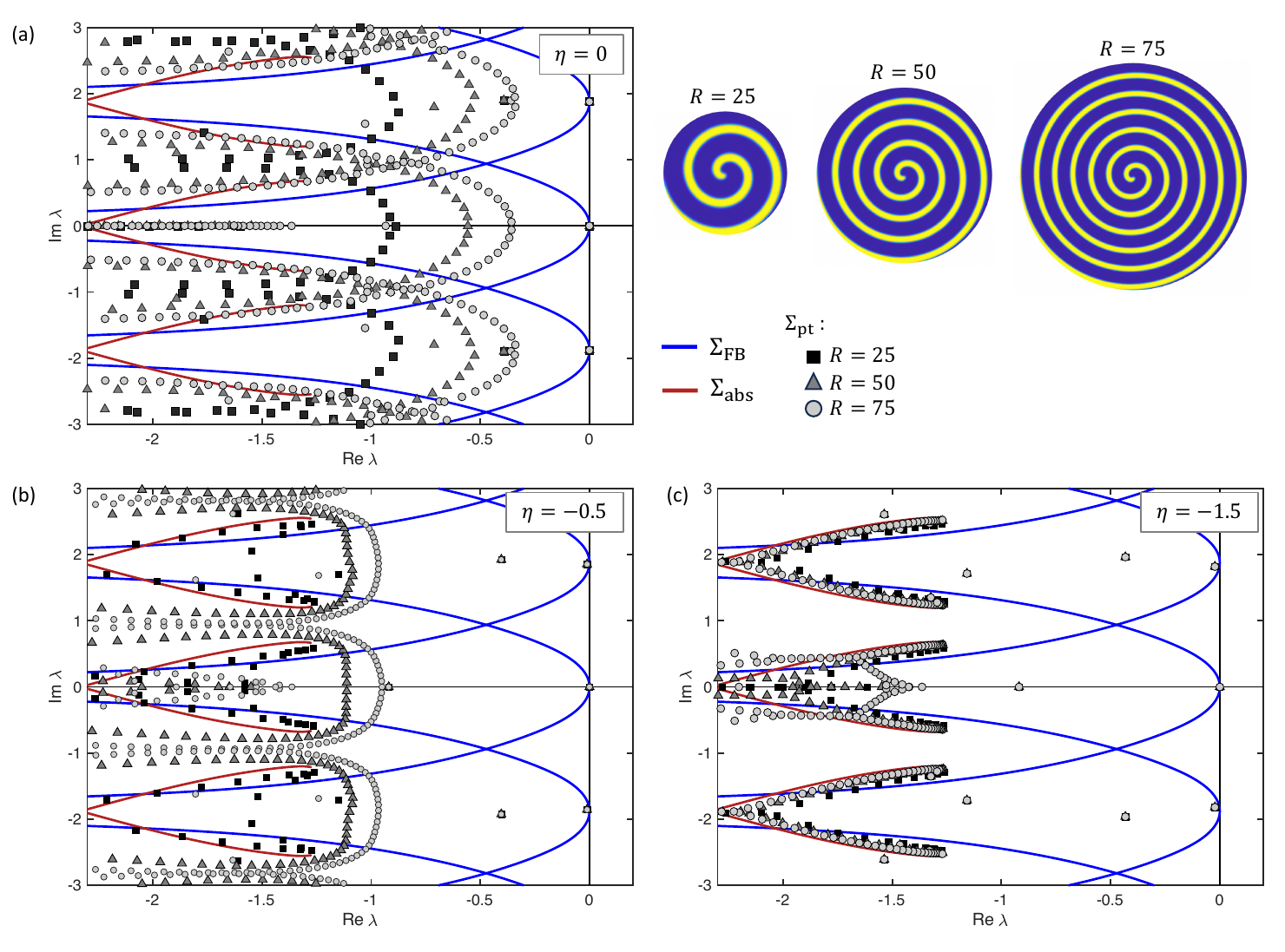}
\caption{(a) Inaccurate computation of point spectra of $\mathcal{L}_R$. Eigenvalues of spiral waves show divergence from $\Sigma_{\text{abs}}$ for increasing $R$ rather than the expected convergence. The spiral profiles capture the $u$-component of the Barkley model. (b)\&(c) Eigenvalues approach the anticipated limit points upon appropriate selection of the exponential weight.}
\label{fig:spec_exp_weight}
\end{figure}

\begin{figure}
\centering
\includegraphics[width=\linewidth]{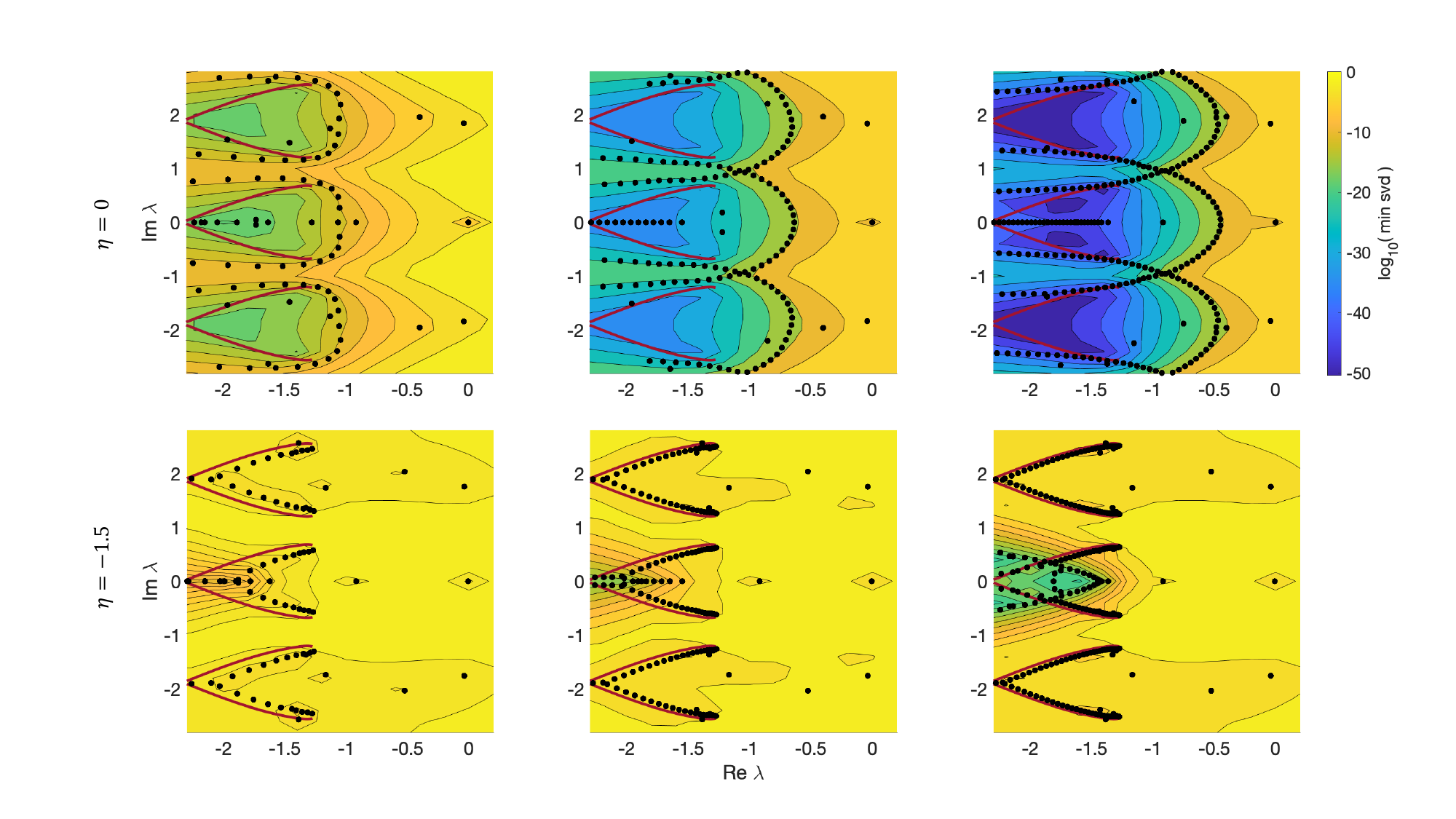}
\caption{Comparison of $\epsilon$-pseudospectra and eigenvalues of the operator $\mathcal{L}_R$ (top row) and  $\mathcal{L}_R^{\eta}$ (bottom row) for $\eta = -1.5$. The three columns correspond, from left to right, to disks of radius $R = 25, 50, 75$. Red curves show $\Sigma_\mathrm{abs}$.}
\label{fig:spiral_pseudospec}
\end{figure}

\begin{figure}
\centering
\includegraphics[width=0.9\linewidth]{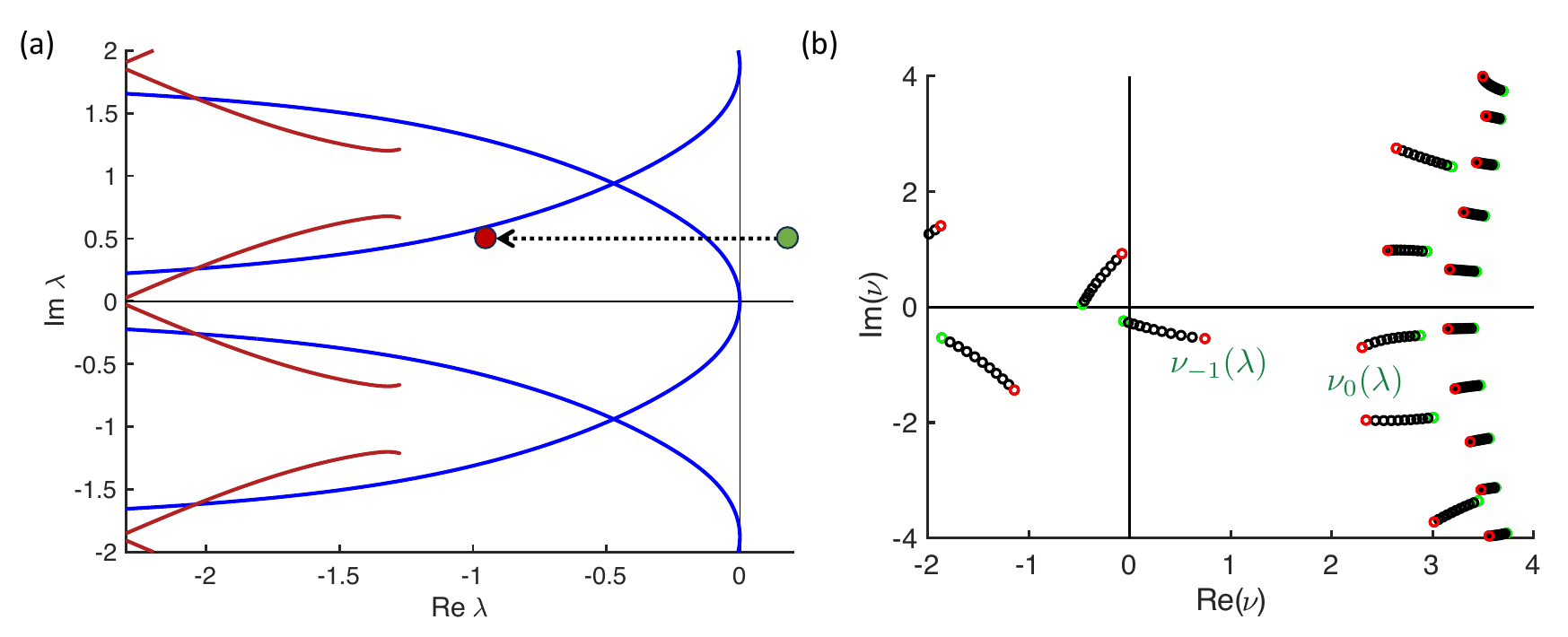}
\caption{Leading spatial eigenvalues $\nu_j(\lambda)$ shown in (b) as $\lambda$ moves along the path indicated by the horizontal dashed arrow in (a). Green and red markers in (b) indicate $\nu_j(\lambda)$ for $\lambda = 0.1 + 0.5 \rmi$ (green) and $\lambda = -1 + 0.5 \rmi$ (red). Spatial eigenvalues $\nu_{-1}(\lambda)$ and $\nu_0(\lambda)$ relevant for $J_0(\lambda)$ are labeled.  }
\label{fig:spatial_eigenvalue_selection}
\end{figure}

\subsection{Application: Barkley model}

The paradigm model
\[\textstyle
u_t = \Delta u + \frac{1}{\epsilon} u (1-u) \left( u - \frac{v+b}{a} \right), \qquad
v_t = \delta \Delta v + u - v
\]
exhibits bifurcations caused by destabilizing spiral-wave spectra \cite{BKT}.

Figures~\ref{fig:spec_exp_weight}a-\ref{fig:spiral_pseudospec} demonstrate that spectral computations for the operator $\mathcal{L}_R$ on the unweighted space yield inaccurate results. As the domain radius increases, eigenvalues in Figure~\ref{fig:spec_exp_weight}a move away from the theoretical absolute-spectrum limit and instead approach curves that resemble the Fredholm boundaries. These inaccurate eigenvalue results arise due to the exponential growth of the resolvent in $R$ over large regions of the complex plane to the left of $\Sigma_{\text{FB}}$, and iterative eigenvalue solvers such as \texttt{eigs} will identify many spurious eigenvalues in these regions. This fact is demonstrated in the top row of Figure~\ref{fig:spiral_pseudospec}, where we observe that the computed eigenvalues align along $\epsilon$-pseudospectrum contours that gradually approach $\Sigma_{\text{FB}}$ as $R$ increases.

Next, consider the operator $\mathcal{L}_R^{\eta}$ in exponentially weighted spaces. The bottom row of Figure~\ref{fig:spiral_pseudospec} contains the $\epsilon$-pseudospectra contours and the computed eigenvalues of the operator $\mathcal{L}_R^{\eta}$ for $\eta=-1.5$: note that the resolvent is better conditioned and eigenvalues are significantly more accurate, and that the only change from the top to bottom rows in Figure~\ref{fig:spiral_pseudospec} is the switch from $\mathcal{L}_R$ to the preconditioned operator $\mathcal{L}_R^{\eta}$.

The selection of exponential weight $\eta$ impacts the eigenvalue accuracy, as displayed in Figure~\ref{fig:spec_exp_weight}. As the exponential weight decreases to $\eta = -1.5$, more eigenvalues of $\mathcal{L}^{\eta}_R$ move closer to the theoretical $R\gg1$ limit $\Sigma_{\text{abs}}$. The choice of $\eta = -1.5$ comes from considering the spatial eigenvalues $\nu_j(\lambda)$. Figure~\ref{fig:spatial_eigenvalue_selection} displays the spatial eigenvalues $\nu_j(\lambda)$ as $\lambda$ moves from $\lambda_1 = 0.1 + 0.5 \rmi$ to $\lambda_2 = -1+ 0.5 \rmi$, that is as $\lambda$ traces out the horizontal path indicated by the dashed line between the green ($\lambda_1$) and red ($\lambda_2)$ markers. As $\lambda$ passes through the $\Sigma_{\text{FB}}$ branch, $\nu_{-1}(\lambda)$ crosses the imaginary axis into the positive half-plane. Theorem~\ref{t:1} suggests weights $\eta \in J_0(\lambda) = \left( - \nu_0(\lambda), -\nu_{-1}(\lambda)\right)$. Thus, for this particular parameter setting in the Barkley model, an exponential weight of $\eta= -1.5$ is a good selection for a large range of $\lambda$ to the left of $\Sigma_{\text{FB}}$. 

\begin{table}
\caption{Shown are the condition numbers $\kappa$ and minimum SVD values for the operator $\mathcal{L}^{\eta}_R - \lambda$ with weight $\eta$ and indicated value of $\lambda$. All values are reported on a $\log_{10}$-scale for radius $R = 75$.}
\begin{center}
\begin{tabular}{|c| c c | c c | c c| }
\toprule
 & \multicolumn{2}{c|} {$\lambda = 0$}  &   \multicolumn{2}{c|}{ $\lambda= -1 + \rmi$ } & \multicolumn{2}{c|}{ $\lambda = -1.5 + \rmi$}\\
 $\eta$ & $\kappa$  & min svd & $\kappa$  & min svd& $\kappa$  & min svd\\
\midrule
0 & 16.4360& -6.4086& 22.2498 & -15.8022& 33.5090 &-27.0790  \\
-0.5 & 13.0844 &-3.6099  & 8.2964 & -1.9682 & 18.1444& -11.8621 \\
-1.0 & 13.5074& -4.2008& 7.5628 & -1.0418& 8.4466& -1.8574 \\
-1.5 & 12.5221&-3.2554 & 8.1392 & -1.3671& 10.4340& -2.3428 \\
-2.0 & 12.2852&-3.1592 & 9.5971 & -2.5870 & 22.7333 & NaN \\
\bottomrule
\end{tabular}
\end{center}
\label{table:condition_numbers}
\end{table}

\begin{figure}
\centering
\includegraphics[width=\linewidth]{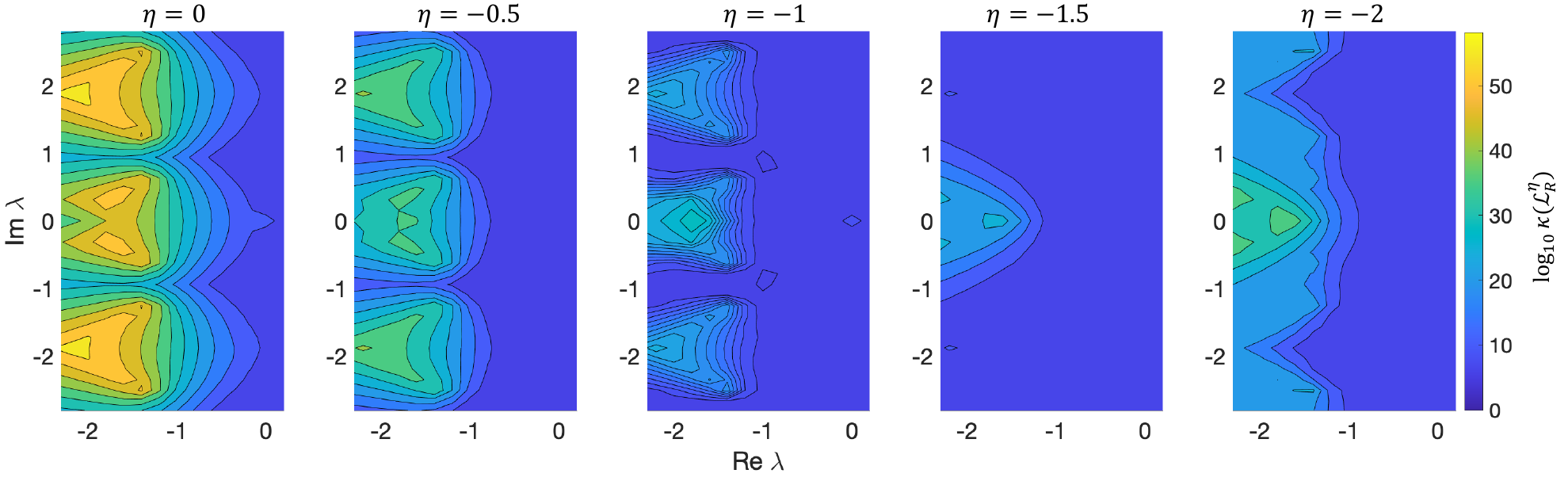}
\caption{Shown is a color plot of the condition numbers of $\mathcal{L}_R^\eta-\lambda$ in a $\log_{10}$-scale in the $\lambda$-plane for radius $R=75$ and various $\eta$.}
\label{fig:condition_number}
\end{figure}

The condition numbers $\kappa$ and the minimum SVD values of the operator $\mathcal{L}^{\eta}_R-\lambda$ of the numerical operator shown in Table~\ref{table:condition_numbers} and illustrated in Figure~\ref{fig:condition_number} demonstrate similar improvement with the addition of the exponential weight. The three selected values for $\lambda$ in Table~\ref{table:condition_numbers} represent points at various distances from $\Sigma_\mathrm{abs}$ and $\Sigma_\mathrm{FB}$. While exponential weights yield only moderate improvements of the condition number for $\lambda$ near the origin (due to the eigenvalue  $0\in\Sigma_\mathrm{ext}^\mathrm{sp}$), appropriate exponential weights improve the condition number by over 25 orders of magnitude for $\lambda$ near $\Sigma_\mathrm{abs}$. Table~\ref{table:condition_numbers} and Figure~\ref{fig:condition_number} also indicate the reduction in efficiency if the weight value is selected outside of $J_0(\lambda)$; weights of $\eta= -2$ result in higher condition numbers than $\eta = -1$ for some $\lambda$.

%%%%%%%%%%%%%%%%%%%%%%%%%%%%%%%%%%%%%%%%%%%%%%%%%%%%%%%%%%%%%%%%%%%%%%%%%

\bibliography{Spiral_Waves_Spectra_Numerics_SIAM}
\bibliographystyle{siamplain}

\end{document}